\documentclass[reqno,12pt]{amsart}

\usepackage{amssymb}
\usepackage{amsthm}

\usepackage{enumitem}

\newtheorem{theorem}{Theorem}[section]
\newtheorem{lemma}[theorem]{Lemma}

\theoremstyle{definition}

\thanks {}

\numberwithin{equation}{section}

\begin{document}

\title[Hyperbolic-type metrics in space]{Hyperbolic-type metrics in space}

\author[Aimo Hinkkanen, Poranee Khayo]{Aimo Hinkkanen$^{1}, $Poranee Khayo$^{2}$}

\address{$^{1}$ Department of Mathematics, University of Illinois Urbana--Champaign,  Urbana, IL, 61801 U.S.A.}

\address{$^{2}$ Department of Mathematics, Physics and Computer Science, University of Cincinnati Blue Ash College,  9555 Plainfield Road, 
Blue Ash, Ohio, 45236 U.S.A.}

\thanks{}


\abstract 
We define metrics in space that are natural counterparts of the hyperbolic metric in plane domains, using the characterization of the hyperbolic metric due to Beardon and Pommerenke. We obtain inequalities for these metrics under quasiconformal homeomorphisms between domains in space that have at least two boundary points. We discuss the failure of the existence of such estimates for non-homeomorphic  quasiregular mappings.
\endabstract

\maketitle

\section{Introduction} \label{s1}

The hyperbolic metric of plane domains whose complement contains at least two points has been used extensively in the theory of analytic functions and quasiconformal mappings in the plane. A metric called the quasihyperbolic metric has been defined in domains in both the plane and space, and it has been used in the theory of quasiconformal and quasiregular mappings in both settings. However, the quasihyperbolic metric does not capture all the properties of the (plane) hyperbolic metric except in simply connected plane domains. 

In this paper we define hyperbolic-type metrics in space and show how they can be used in the theory of quasiconformal mappings. In the plane, our definition is a natural consequence of the results of Beardon and Pommerenke \cite{BP78}.
But that metric fails to be monotone with respect to domain (a property of the hyperbolic metric), and so we consider three variants of the metric and show that they are comparable within a multiplicative positive absolute constant and hence could often be used interchangably according to whatever is most convenient for the situation. One of the three is monotone with respect to domain. Aspects of the results of Beardon and Pommerenke in plane domains have been previously developed further at least by Herron, Ma, and Minda \cite{HMM08} and particularly by Herron and Lindquist \cite{HL21}. 

The plan of the paper is as follows. In Section~\ref{s2} we define the three metrics and state their principal properties, which will be proved in Sections~\ref{s4}--\ref{pf2}. In Section~\ref{s3} we  state our results for quasiconformal mappings obtained by means of these metrics and discuss the situation for  quasiregular mappings. 
In Sections~\ref{s5} and \ref{s6} we prove the results stated in Section~\ref{s3}.

{\bf Notations.} 
We deal with ${\mathbb R}^n$, where $n\geq 2$.

If $x\in {\mathbb R}^n$ and $0<r_1<r_2$, we write
\begin{equation} \label{ann}
A(x,r_1,r_2) = \{ y\in {\mathbb R}^n : r_1<|y-x|<r_2 \}
\end{equation}
for an annulus (if $n=2$) or a spherical ring domain (if $n\geq 3$).
We also write, for $x\in {\mathbb R}^n$ and $r>0$, 
$$
B^n(x,r)=\{y\in {\mathbb R}^n : |y-x|<r \},\quad 
\overline{B}^n(x,r)=\{y\in {\mathbb R}^n : |y-x|\leq r \},
$$
$$
S^{n-1}(x,r)=\{y\in {\mathbb R}^n : |y-x|=r \} .
$$
The $n-$dimensional Lebesgue measure of $B^n(x,r)$ is denoted by $|B^n(x,r)|$. The phrase almost everywhere refers to the Lebesgue measure. Integration with respect to the Lebesgue measure is denoted by $dm(x)$. 
If $x=(x_1,\dots , x_n)\in {\mathbb R}^n$, we write
$$
|x| = \sqrt{ x_1^2 + \cdots + x_n^2} .
$$
We denote the standard basis of ${\mathbb R}^n$ by $\{ {\bf e}_1, \dots , {\bf e}_n \}$. For $x=(x_1,\dots , x_n)\in {\mathbb R}^n$ and $y=(y_1,\dots , y_n)\in {\mathbb R}^n$, we write
$$
x\cdot y = \sum_{j=1}^n x_jy_j  .
$$

For the proof of Lemma~\ref{lem2}, we note that there is a number 
\begin{equation} \label{t00}
t_0\approx  1.14619 
\end{equation} 
such that $e^{t_0}=2+t_0$, which we will need in the form
$$
e^{t_0}-1 = t_0+1 .
$$ 
For $0\leq t<t_0$, we have
$e^t<2+t$. For $t>t_0$, we have $e^t>2+t$.  These numerical calculations have been performed by Wolfram's  software package  {\sl Mathematica}.

\section{The three metrics} \label{s2}

Suppose that $n\geq 2$, and let $D$ be a domain in the Euclidean space ${\mathbb R}^n$ such that ${\mathbb R}^n \setminus D$ contains at least two points. For $z\in D$, define
$$
d(z,\partial D) = \inf \{ |z-a| : a\in \partial D\} .
$$
There is at least one point $a\in \partial D$ such that $|z-a| = d(z,\partial D)$.

For each $z\in D$, we define three densities of metrics, denoted by $\lambda_D(z)$, $\lambda'_D(z)$,  and $\lambda''_D(z)$, as follows.

We set
\begin{equation} \label{met1}
 \frac{ 1}{\lambda_D(z) } = \inf\left\{  |z-a| \left(  1 + \left| \log \frac{  |a-b|    } {  |z-a|   }  \right|  \right)  : a,b\in {\mathbb R}^n \setminus D      \right\} ,
\end{equation}
\begin{equation} \label{met2}
 \frac{ 1}{\lambda'_D(z) } = \inf\left\{  |z-a| \left(  1 + \left| \log \frac{  |a-b|    } {  |z-a|   }  \right|  \right)  : a,b\in \partial D     \right\} ,
\end{equation}
and
\begin{equation} \label{met3}
 \frac{ 1}{\lambda''_D(z) } = \inf\left\{  |z-a| \left(  1 + \left| \log \frac{  |a-b|    } {  |z-a|   }  \right|  \right)  : a,b\in \partial D, \,\, |z-a|= d(z,\partial D)      \right\} .
\end{equation}

Since ${\mathbb R}^n \setminus D$, and hence also $\partial D$, contains at least two points, it is clear in each case that the density is a finite positive number.

We see that the main idea is to consider the infimum of the numbers 
$$
 |z-a| \left(  1 + \left| \log \frac{  |a-b|    } {  |z-a|   }  \right|  \right) 
$$
where $a$ and $b$ are outside $D$. The only question is exactly to which sets $a$ and $b$ are limited. There are other variations than the three mentioned above. For example, we could choose $b\in {\mathbb R}^n \setminus D$ and $a\in \partial D$, with or without the condition that $|z-a|= d(z,\partial D)$.

One possible motivation for considering such metrics is the following result due to Beardon and Pommerenke from 1978 (\cite{BP78}, Theorem~1, p.~477). We denote the density of the hyperbolic metric of a plane domain $D$ at $z\in D$ by $\eta_D(z)$ and normalize it so that when $D$ is the unit disk ${\mathbb D} = \{ z\in {\mathbb C} \colon |z|<1 \}$ in the complex plane ${\mathbb C}$, which we identify with ${\mathbb R}^2$, then $\eta_{ {\mathbb D} }(z) = 1/(1-|z|^2)$ for all $z\in {\mathbb D}$. 

\noindent {\bf Theorem A.} {\sl  Let $D$ be a plane domain with a hyperbolic metric with density $\eta_D(z)$. Write $k=4+\log (3+2\sqrt{2})\approx 5.7627...$. For all $z\in D$, define
\begin{equation} \label{beta1}
\beta(z) = \inf\left\{   \left| \log \frac{  |a-b|    } {  |z-a|   }  \right|    : a,b\in \partial D, \,\, |z-a|= d(z,\partial D)      \right\}  .
\end{equation}
Then for every $z\in D$, we have
\begin{equation} \label{hyp1}
\frac{1}{2 \sqrt{2}  } - k\cdot \eta_D(z) \cdot d(z,\partial D) \leq
\eta_D(z) \cdot d(z,\partial D) \cdot \beta(z) \leq k+ \frac{\pi}{4}
\end{equation}
and also
\begin{equation} \label{hyp2}
\frac{1}{2\sqrt{2}}  \leq
\eta_D(z) \cdot d(z,\partial D) \cdot (\beta(z) + k) \leq k+ \frac{\pi}{4}  .
\end{equation}
}

If $D$ is simply connected, then $\beta(z)=0$ for all $z\in D$. The above result shows that in a plane domain, if we do not care about the exact values of the constants, the quantity $1/\eta_D(z)$ is comparable to 
\hfil\break
$d(z,\partial D) \cdot (\beta(z) + k)$ and hence also to $d(z,\partial D) \cdot (\beta(z) + 1)$. This last quantity is equal to $1/\lambda''_D(z)$. These observations served as our starting point for the investigations in this paper.

In domains in space, a hyperbolic metric has been defined only in balls, exteriors of balls, and half spaces. We are thus defining what is a hyperbolic-type metric in general domains in space. 

To continue with our study of these metrics in space, we first establish some simple properties of their densities. We assume throughout that $D$ is a domain in  ${\mathbb R}^n$, where $n\geq 2$, such that ${\mathbb R}^n \setminus D$ contains at least two points. 

\begin{lemma} \label{lem1}
Each infimum in (\ref {met1}), (\ref {met2}), and (\ref {met3})  is attained as a minimum. Each of $\lambda_D(z)$, $\lambda'_D(z)$, and $\lambda''_D(z)$ is a continuous function of $z$ for $z\in D$.
\end{lemma}

\begin{lemma} \label{lem2}
There is a positive absolute constant $C_0$ such that whenever $n\geq 2$, $D$ is a domain in ${\mathbb R}^n$ such that ${\mathbb R}^n \setminus D$ contains at least two points, and $z\in D$,  we have
\begin{equation} \label{comp1}
 \lambda''_D(z) \leq \lambda'_D(z) \leq \lambda_D(z) \leq C_0 \lambda''_D(z) 
\end{equation} 
and
\begin{equation} \label{comp2}
 \lambda_D(z) \leq  \frac{1}{ d(z,\partial D) }  .
\end{equation} 
\end{lemma}

Strict inequalities may hold in (\ref{comp1}). 

\begin{lemma} \label{le5}
For each $n\geq 2$, there is a domain $D$  in ${\mathbb R}^n$ such that ${\mathbb R}^n \setminus D$ contains at least two points and such that there exists $z\in D$ for which
\begin{equation} \label{comp3a}
 \lambda''_D(z) < \lambda'_D(z) < \lambda_D(z) .
\end{equation}
\end{lemma}

The hyperbolic metric of plane domains is monotone with respect to the domain. More precisely, if $D_1$ and $D_2$ are plane domains with $D_1\subset D_2$, if ${\mathbb C}\setminus D_2$ contains at least two points, and if $z\in D_1$, then
$$
\eta_{D_2}(z)\leq \eta_{D_1}(z) .
$$
The densities $\lambda'_D(z)$ and $\lambda''_D(z)$ do not have this property. The reason for considering $\lambda_D(z)$ is that it has this property. The following Lemma~\ref{le3} is obvious on the basis of the definition of $\lambda_D(z)$. 

\begin{lemma} \label{le3}
Suppose that $n\geq 2$ and that $D_1$ and $D_2$ are domains in ${\mathbb R}^n$ with $D_1\subset D_2$, such that ${\mathbb R}^n \setminus D_2$ contains at least two points. 

If $z\in D_1$, then
\begin{equation} \label{comp3}
 \lambda_{D_2}(z) \leq   \lambda_{D_1}(z)  .
\end{equation}
\end{lemma}

\begin{lemma} \label{le4}
There exist domains $D_1$ and $D_2$  in ${\mathbb R}^2$ with $D_1\subset D_2$, such that ${\mathbb R}^2 \setminus D_2$ contains at least two points, and such that for at least one point $z\in D_1$, we have 
\begin{equation} \label{comp4}
  \lambda_{D_2}(z) = \lambda''_{D_2}(z) = \lambda'_{D_2}(z)>   \lambda'_{D_1}(z)
  >  \lambda''_{D_1}(z)  
\end{equation}
and in particular $ \lambda'_{D_2}(z)>   \lambda'_{D_1}(z)$ and $ \lambda''_{D_2}(z)   >  \lambda''_{D_1}(z)  $.  
\end{lemma}

For $z,w\in D$, we define distances by 
$$
d_D(z,w) = \inf \int_{\gamma} \lambda_D(t) \, |dt| ,
$$
$$
d'_D(z,w) = \inf \int_{\gamma} \lambda'_D(t) \, |dt| ,
$$
and
$$
d''_D(z,w) = \inf \int_{\gamma} \lambda''_D(t) \, |dt| ,
$$
where the infima are taken over all rectifiable arcs $\gamma$ that join $z$ to $w$ in $D$. The following result is implied by Lemma~\ref{lem2}.

\begin{lemma} \label{le6}
Let $C_0$ be the positive absolute constant in Lemma~\ref{lem2}. Suppose that $n\geq 2$ and let $D$ be a domain in ${\mathbb R}^n$ such that ${\mathbb R}^n \setminus D$ contains at least two points. Then, whenever $z,w\in D$,  we have
\begin{equation} \label{comp6}
 d''_D(z,w) \leq d'_D(z,w) \leq d_D(z,w) \leq C_0 d''_D(z,w) .
\end{equation} 
\end{lemma}

We also see that for all $z,w\in D$ we have $d_D(z,w) \leq k_D(z,w)$, where analogously to the above, the well known quasihyperbolic metric is defined by
$$
k_D(z,w) =  \inf \int_{\gamma} \frac{  |dt|  } {d(t,\partial D)} \,    .
$$

We will also need the following technical lemma, which may sometimes be of independent interest.

\begin{lemma} \label{le7}
Fix $z\in D$. Let $(a',b')$ be any pair $(a,b)$ for which the infimum in (\ref{met1}) for $\lambda_D(z) $ is attained as a minimum. Then $a' \in \partial D$. Furthermore,  $b' \in \partial D$ except possibly if $|z-a'| = |a'-b'|$, in which case we may have $b'\in {\mathbb R}^n \setminus \overline{D}$, but even then there is a pair $(a'',b'')$ with $a'',b''\in \partial D$ for which the infimum is attained, except possibly in the following case: we have $a'=(z+b')/2$ for every minimizing pair $(a',b')$, and the distance $|z-a'|$ is the same for every such pair. When we are not in this exceptional case, we have $\lambda_D(z) = \lambda'_D(z)$, while in the exceptional case we have
$\lambda_D(z) > \lambda'_D(z)$. 
\end{lemma}

\section{The metrics and quasiconformal and quasiregular mappings} \label{s3}

\subsection{Definitions.} We recall the following definitions from \cite{r5}. Suppose that $n\geq 2$, and let $D$ be a domain in the Euclidean $n$-space 
${\mathbb R}^n$.  Let $f :D\to {\mathbb R}^n$ be a non-constant  function.
We say that $f$ is quasiregular  if $f$ is continuous and belongs to the Sobolev space $W^1_{n,{\rm loc}}(D)$, and if there is is a real number $K$ with $K\geq 1$ such that 
\begin{equation} \label{qr-def}
| f'(x) |^n \leq K J_f(x) 
\end{equation} 
for almost every $x\in D$. Here $| f'(x) |$ denotes the norm of the formal derivative mapping of $f$ at $x$, and $J_f(x)$ denotes the Jacobian determinant (the determinant of the matrix $f'(x)$) of $f$ at $x$. Under the assumptions above, $f$ is differentiable almost everywhere (a.e.) in $D$. The infimum of the numbers $K$ for which (\ref{qr-def}) holds a.e.\ in $D$ is called the outer dilatation $K_O= K_O(f)$ of $f$. If $f$ is quasiregular, then there is a number $K'\geq 1$ such that 
\begin{equation} \label{qr-inner}
 J_f(x) \leq K' \inf\{  |f'(x)h| \colon h \in {\mathbb R}^n, \,\, |h|=1        \} 
\end{equation} 
for a.e.\ $x\in D$. The infimum of such numbers $K'$ is called the inner dilatation $K_I= K_I(f)$ of $f$. The number $K(f)=\max\{K_O(f),K_I(f)\}$ is called the maximal dilatation of $f$. We say that $f$ is $K-$quasiregular is $K(f)\leq K$. 
Thus $f$ is quasiregular if $f$ is $K-$quasiregular  for some $K\geq 1$. (Constant functions $f :D \to {\mathbb R}^n$ are also called quasiregular, but all functions that we consider will be non-constant.)

Let $D$ and $D'$ be domains in ${\mathbb R}^n$ and let  $f$ be a homeomorphism of $D$ onto $D'$. If $f$ is also quasiregular, we say that $f$ is quasiconformal, and if $f$ is $K-$quasiregular, we say that $f$ is $K-$quasiconformal. Then $f^{-1}$ is also 
$K-$quasiconformal (\cite{V71}, p.~42). 

\subsection{Metrics and quasiconformal mappings.}

\begin{theorem} \label{th-qc1}
Suppose that $n\geq 2$, and let $D$ and $D'$ be domains in ${\mathbb R}^n$, each with at least two finite boundary points. 

Let $f$ be a $K-$quasi\-conformal mapping of $D$ onto $D'$. Then there is a constant $C_1>1$ depending on $n$ and $K$ only such that for every $z\in D$ we have
\begin{equation} \label{qc1}
\frac{1}{C_1} \leq   \frac{ \lambda_D(z)  d(z,\partial D) } { \lambda_{D'}(f(z))  d(f(z),\partial D' )   }    \leq C_1  .
\end{equation}
\end{theorem}
Because of Lemma~\ref{lem2}, it follows that  similar results are valid with $\lambda$ replaced by $\lambda'$ or by $\lambda''$. 

We use Theorem~\ref{th-qc1} to obtain the following estimate. We denote by $\alpha\in (0,1]$ the quantity $1/K^{1/(n-1)}$.

\begin{theorem} \label{th-qc2}
Suppose that $n\geq 2$, and let $D$ and $D'$ be domains in ${\mathbb R}^n$, each with at least two finite boundary points. Let $f$ be a $K-$quasi\-conformal mapping of $D$ onto $D'$. Then there is a constant $C_2>1$ depending on $n$ and $K$ only such that for all $z,w\in D$ we have
\begin{equation} \label{qc2-1}
  d_{D'}(f(z) , f(w) )   \leq C_2 \max \{ d_D(z,w) , d_D(z,w) ^{\alpha} \}  
\end{equation}
and
\begin{equation} \label{qc2-2}
 d_D(z,w)  \leq C_2  \max \{ d_{D'}(f(z) , f(w) ) , d_{D'}(f(z) , f(w) )^{\alpha}  \}  .
\end{equation}
\end{theorem}

It is sufficient to prove (\ref{qc2-1}), for if $f$ is $K-$quasiconformal, then $f^{-1}$ is  also $K-$quasiconformal, and then we may apply  (\ref{qc2-1}) to $f^{-1}$ to obtain (\ref{qc2-2}).

Because of Lemma~\ref{lem2}, it follows that a similar result is valid with $d_D(z,w)$ replaced by $d'_D(z,w)$ or by $d''_D(z,w)$.

\subsection{Remarks on quasiregular mappings.}

For the purpose of a general discussion, let us call the metric  densities $\lambda_D$, $\lambda'_D$, and $\lambda''_D$ briefly $\lambda-$metrics.

In a plane domain $D$ with at least two boundary points, the $\lambda-$metrics are comparable to the density $\eta_D$ of the hyperbolic metric of $D$ by the result of Beardon and Pommerenke \cite{BP78}. 

If $D$ and $D'$ are plane domains with a hyperbolic metric and $f$ is an analytic function of $D$ into $D'$, then
\begin{equation} \label{hyp33}
\eta_{D'} ( f(z) )  |f'(z)| \leq \eta_D(z)
\end{equation}
for all $z\in D$, with equality if $f$ is a conformal mapping of $D$ onto $D'$. Comparison to (\ref{qc1}), ignoring multiplicative absolute constants in this general discussion, shows that we have been able to replace $ |f'(z)|  $ by
$d(f(z),\partial D')/d(z,\partial D)$. For conformal $f$ this means that
$$
\frac{  |f'(z)|  } {   d(f(z),\partial D')/d(z,\partial D)      } 
$$
lies between positive absolute constants, which is a well-known and often used consequence of the Koebe $1/4-$theorem.  

By Theorem~\ref{th-qc1}, the same replacement is possible also for quasiconformal 
homeomorphisms in all dimensions $\geq 2$. There is a natural reason for this. By (\ref{met3}), we have
$$
\frac{1} {    \lambda''_D(z)     d(z,\partial D)          } = 
 \inf\left\{ \left(  1 + \left| \log \frac{  |a-b|    } {  |z-a|   }  \right|  \right)  : a,b\in \partial D, \,\, |z-a|= d(z,\partial D)      \right\} .
$$
Therefore Theorem~\ref{th-qc1} states that if one of $D$ and $D'$ contains an annulus or spherical ring domain of large modulus in a suitable position, then so does the other one of $D$ and $D'$. Since quasiconformal 
homeomorphisms change the moduli of ring domains by a bounded multiplicative factor only, it is natural that this should be true.

If $D$ and $D'$ are bounded domains in ${\mathbb R}^n$ with connected boundary, then $ \lambda''_D(z) = 1/ d(z,\partial D)   $ and similarly in $D'$, so that the $\lambda-$metrics coincide with the quasihyperbolic metric. For some other domains, such as those whose boundary is uniformly perfect (see \cite{P} or \cite{JV} for the definition), the $ \lambda-$metrics might be comparable to the quasihyperbolic metric by multiplicative constants, due to the nature of the boundaries of $D$ and $D'$.  In such situations known results for (also non-homeomorphic) quasiregular mappings would immediately apply also to the $ \lambda-$metrics due to this conincidence, such as Theorems~11.5, 12.5, and  12.21 (and many others) in \cite{Vuo88}, some of which also have necessary additional assumptions in the form of Harnack conditions.  

For non-homeomorphic quasiregular mappings, a counterpart of (\ref{hyp33}) is not readily available and not always even true, as we will show by examples in a moment. The image of a ring domain under a (non-constant) non-homeomorphic quasiregular mapping is merely a subdomain of $D'$, usually not a ring domain, so that arguments in the proof of Theorem~\ref{th-qc1} break down already for that reason. When looking for a replacement of $|f'(z)|$ in terms of quantities related to derivatives, we note that $f'(x)$ usually does not exist for all $x\in D$. An average such as 
$$
\frac{1}{ | B^n(x, \theta d(x,\partial D) ) | } \int_{   B^n(x, \theta d(x,\partial D) )  } |f'(y)| \, dm(y) 
$$
for a fixed small $\theta\in (0,1)$, can be seen to be too large to work as a replacement even in (\ref{hyp33}) by considering the function $f(z)=z^k$ of the unit disk into itself for arbitrarily large positive integers $k$. A limit quantity such as
$$
\liminf_{ \theta \to 0} \frac{1}{ | B^n(x, \theta d(x,\partial D) ) | } \int_{   B^n(x, \theta d(x,\partial D) )  } |f'(y)| \, dm(y) 
$$
or
$$
\liminf_{ \theta \to 0} \left(   \frac{1}{ | B^n(x, \theta d(x,\partial D) ) | } \int_{   B^n(x, \theta d(x,\partial D) )  } |f'(y)|^n \, dm(y) \right)^{1/n}
$$
may be infinite, as seen by considering the quasiconformal 
homeomorphism $x\mapsto x |x|^{1/K-1}$ at the origin, where $K>1$. 

Choose $D'=B^2(0,1)$ in the plane, a small $\varepsilon >0$, and $f(z)=z^2+ \varepsilon z$. Then $D_1=f^{-1}(D')$ is almost the same as the unit disk. Define $D=D_1 \setminus \{0\}$. Since $f(-\varepsilon)=0$, $f$ maps $D$ onto $D'$. Close to the origin, the density of the hyperbolic metric of $D$ is close to that of the punctured disk $B'=B^2(0,1)\setminus \{0\}$, and $\eta_{B'}(z) = 1/( 2 |z| \log (1/|z|))$. 
Hence close to the origin, $ \eta_{B'}(z) d(z,\partial D) \approx 1/( 2  \log (1/|z|) ) $
while (since $f(z)$ is close to zero) $ \eta_{D'}(f(z)) d(f(z),\partial D') \approx 1$.  
Hence it is not true that
$$
\frac{  \eta_{D'}(f(z)) d(f(z),\partial D')   } {    \eta_{D}(z) d(z,\partial D)    } 
\approx 2 \log (1/|z|) 
$$
remains below a fixed constant as $z\to 0$. The image under $f$ of an annulus $A(0,|z|,1/2)$, say, has no particular relation to the nearest boundary point of $D'$ to $f(z)$ at least in the sense that the inner boundary of $f( A(0,|z|,1/2)  )$ is not close to $\partial D'$. 

One might ask if the situation were to become better if $D$ were to have a connected boundary and $D'$ were to be a punctured sphere, but this is not so simple either as we shall see.

Theorem~\ref{th-qc1} was used as a tool to prove Theorem~\ref{th-qc2}. Now Theorem~\ref{th-qc2} does not mention any derivatives or the distance to the boundary, so one can ask whether a form of Theorem~\ref{th-qc2} could be proved for non-homeomorphic quasiregular mappings in some other way, without a counterpart of Theorem~\ref{th-qc1}. 

For quasiregular mappings between plane domains $D$ and $D'$ with a hyperbolic metric, it is well-known how to do this. We recall the details for completeness. Let $f$ be a non-constant $K-$quasiregular mapping of $D$ into $D'$. By Sto\"{i}low factorization, there is a quasiconformal 
homeomorphism $\varphi$ of $D$ onto a domain $D''$ and an analytic function $g$ of $D''$ into $D'$ such that $f=g\circ \varphi$. Let $h_D(z,w)$ denote the hyperbolic distance of $z,w\in D$ with respect to the hyperbolic metric of $D$. Since
$h_{D'} (g(z),g(w)) \leq h_{D''} (z,w)$ for all $z,w\in D''$, it suffices to estimate
$h_{D''} (\varphi (z),\varphi (w))$ in terms of $h_D(z,w)$ for $z,w\in D$. 

Pick $z_0,w_0\in D$. Let $\psi$ be an analytic universal covering map of $B^2(0,1)$ onto $D$ with $\psi(0)=z_0$. Let $\omega$ be an analytic universal covering map of $B^2(0,1)$ onto $D''$ with $\omega(0)=\varphi( z_0)$. Since $\varphi$ is a homeomorphism, the monodromy theorem applies and we can continue a branch of $F=\omega^{-1} \circ \varphi \circ \psi$ with $F(0)=0$ to a single-valued $K-$quasiregular mapping of $B^2(0,1)$ into itself, still denoted by $F$. Let $w'\in B^2(0,1)$ be the point closest to the origin with $\psi(w')=w_0$ and let $w''=F(w')\in B^2(0,1)$ so that $\omega(w'')=\varphi( w_0 ) $. Then 
$h_{B^2(0,1) } (0,w') = h_D (z_0,w_0)$ and $h_{B^2(0,1) } (0,w'') \geq h_{D''} (\varphi( z_0), \varphi( w_0) )$.  By Corollary~11.3 in \cite{Vuo88}, p.~138, we have
$|w''| = |F(w') | \leq 2^{1-1/K} K |w'|^{1/K}$. This can be used to provide an upper bound for $h_{B^2(0,1) } (0,w'')$ in terms of $h_{B^2(0,1) } (0,w') $, as required. 

Let us now consider non-homeomorphic quasiregular mappings in dimensions $n\geq 3$. Most domains in ${\mathbb R}^n$ with $n\geq 3$ do not have a hyperbolic metric but the unit ball $B^n(0,1)$ has such a metric for with we use the same notations as in the plane. As usual we compactify ${\mathbb R}^n$ by setting $\overline{ {\mathbb R}^n} = {\mathbb R}^n \cup \{\infty\}$ and identify $\overline{ {\mathbb R}^n} $ with the sphere ${\mathbb S}^n = S^n(0,1)$ in ${\mathbb R}^{n+1}$. Suppose that $q\geq 3$ and set $Y={\mathbb S}^n \setminus \{a_j \colon 1\leq j\leq q \}$ where the points $a_j\in {\mathbb S}^n$ are all distinct. 

Let $f$ be a non-constant 
$K-$quasiregular mapping of ${\mathbb R}^n$ into $Y$. Replacing $f$ by $M\circ f$ for a M\"{o}bius transformation $M$ does not change $K$, so we assume that $a_1=\infty$ and hence $Y= {\mathbb R}^n  \setminus \{a_j \colon 2\leq j\leq q \}$. Then the $\lambda-$metrics can be defined in $Y$. 

Suppose that there are {\sl any} distance functions (metrics) $d_1(z,w)$ in $B^n(0,1)$ and $d_2(z,w)$ in $Y$ such that for all $z_0\in Y$ we have $d_2(z_0,z) \to\infty$ as $z$ tends to any boundary point of $Y$, any constant $A\geq 0$, and any increasing homeomorphism $H$ of $[0,\infty)$ onto itself (so $H(0)=0$; one example is $H(t)=C \max\{t, t^{\alpha} \}$ for  constants $C>0$ and $\alpha\in (0,1]$)) such that
for all $K-$quasiregular mappings $g$ of $B(0,1)$ into $Y$ and for all $z,w\in B^n(0,1)$ we have 
\begin{equation} \label{se1}
d_2( g(z),g(w) ) \leq H( d_1(z,w) ) + A .
\end{equation} 
Choose arbitrary $x,y\in {\mathbb R}^n$. For each $R>\max\{|x|,|y|\}$ we define
$g_R(z)=f(Rz)$ so that $g_R$ is a $K-$quasiregular mapping of $B^n(0,1)$ into $Y$.
For all $R>\max\{|x|,|y|\}$ we have $g_R(x/R)=f(x)$ and $g_R(y/R)=f(y)$. As $R\to\infty$, we have
$d_1(x/R,y/R)\to 0$ so that $H( d_1(x/R,y/R) ) + A  \to A$. If (\ref{se1}) holds, it follows that for all $x,y\in {\mathbb R}^n$ we have $d_2( f(x),f(y) ) \leq  A$. So $f$ is a bounded quasiregular mapping in ${\mathbb R}^n$, hence constant (\cite{r5}, Corollary III.1.14, p.~64). 

Suppose that $n\geq 3$ and that distinct points $a_1,\dots , a_q\in {\mathbb S}^n$ are given. Rickman \cite{r4} for $n=3$ and Drasin and Pankka \cite{DP} for all $n\geq 3$ have proved that there is then a large enough $K$ (depending on the given data) such that there exists a non-constant $K-$quasiregular mapping of ${\mathbb R}^n$ into $Y$. The argument in the preceding paragraph shows that in such a situation (\ref{se1}) cannot be valid, no matter how we choose distance functions $d_1$ and $d_2$, the homeomorphism  $H$, and the constant $A\geq 0$. Presumably this observation is well-known. 

In the other direction, Rickman \cite{r1} proved that when $n\geq 3$ and $K>1$ are first given, there is a positive integer $q_0(n,K)$ depending on $n$ and $K$ only such that if above $q\geq q_0(n,K)$, then there is no non-constant $K-$quasiregular mapping of ${\mathbb R}^n$ into $Y$. In that situation it is conceivable that an inequality of the form  (\ref{se1}) might hold. In \cite{r3} where Rickman gave another proof of this result by using certain metrics, Rickman obtained a form of (\ref{se1}). We now recall the details.

Rickman  \cite{r3} considered spherical rather than chordal distances but this does not make any essential difference. Let $\sigma(z,w)$ denote the spherical distance of $z,w\in {\mathbb S}^n $. Suppose that
$$
0< \beta \leq (1/4) \min\{ \sigma(a_j,a_k) \colon 1\leq j<k\leq q       \}    
$$
and define for $1\leq j\leq q$
$$
U_j= \{ y\in {\mathbb S}^n \colon 0< \sigma(y,a_j) < \beta \} \subset Y 
$$
and $U=\bigcup_{1\leq j\leq q} U_j$. 
Rickman considered distance functions $\tau(x,y)$ on $Y$ that satisfy for some fixed positive constants $P$ and $Q$
\begin{equation} \label{tau}
\left|  \tau(x,y) -        \left| \log \frac{ \log (  1/\sigma(  a_j,x   )        )  }    { \log (  1/\sigma(  a_j,y   )        ) }            \right|       \right|  \leq P 
\end{equation}
whenever $x,y\in U_j$, $1\leq j\leq q$,  and 
$$
 \tau(x,y) \leq Q \,  \sigma(x,y) 
$$
whenever $x,y\in Y \setminus U$. The diameter of $Y \setminus U$  in the $\tau-$metric is finite, so it is easily seen that estimates of the general form (\ref{tau}), suitably modified, hold also if $x \in Y \setminus U$ and $y\in U_j$ for some $j$, or if $x\in U_j$ and $y\in U_k$, where $j\not= k$.
Each of the $\lambda-$metrics satisfies these conditions if we use spherical rather than Euclidean distances in the definitions.

Rickman  \cite{r3} proved that there is a positive constant $\delta(n,K)>0$ such that if $q\geq q_0(n,K)$ and $f$ is a non-constant $K-$quasiregular mapping of $B^n(0,1)$ into $Y$, then for all $x,y\in B^n(0,1)$ we have
$$
\tau (f(x),f(y) ) \leq C \max\{ h_{B^n(0,1) } (x,y) , \delta(n,K) \} , 
$$
which (since $\max\{ h_{B^n(0,1) } (x,y) , \delta(n,K) \} \leq h_{B^n(0,1) } (x,y) + \delta(n,K)$) can be viewed as a result of  type (\ref{se1}).  Here $C$ depends on $n$, $K$, $\beta$, $P$, and $Q$. 

Rickman (\cite{r2}, p.~235) suggested that if $q\geq q_0(n,K)$, there is the stronger result
$$
\tau (f(x),f(y) ) \leq C \max\{ h_{B^n(0,1) } (x,y) , h_{B^n(0,1) } (x,y)^{\alpha} \} , 
$$
where $\alpha = K_I^{1/(1-n)}$, but this remains an open question. We thank Matti Vuorinen for pointing out this open problem to us. It follows that the question of the validity of inequalities of the type (\ref{se1}) for quasiregular mappings is a more complicated problem that cannot be addressed by considering metrics only. 

\section{Proof of the continuity properties of the three metrics} \label{s4}

In this Section, we prove Lemma~\ref{lem1}.

\subsection{Proof of Lemma~\ref{lem1}.} 

{\bf The infimum for $1/\lambda'_D(z)$  is a minimum.}
Let the assumptions of Lemma~\ref{lem1}  be satisfied and suppose that $z\in D$. Let us first observe that the infimum in (\ref{met2}) is attained as a minimum. For all $a,b\in \partial D$, we have
$$
 |z-a| \left( 1 +  \left|  \log \frac { |a-b|  }  {  |z-a|  }     \right| \right) \geq  |z-a| \geq {\rm dist}\, (z,\partial D) >0
 $$
 so that the infimum in (\ref{met2}) is positive and indeed is at least ${\rm dist}\, (z,\partial D)$.  
There are sequences $a_j,b_j\in \partial D$ such that
$$
 |z-a_j| \left( 1 +  \left|  \log \frac { |a_j-b_j|  }  {  |z-a_j|  }     \right| \right)   \to 1/\lambda'_D(z) .
$$
Then, firstly, the sequence $ |z-a_j| $ is bounded and hence the sequence $a_j$ lies in a compact subset of $\partial D$. Next, it follows that the sequence $b_j$ is bounded also. 
Since $\partial D$ is a closed set, we may pass to subsequences without changing notation and assume that $a_j\to a\in\partial D$ and $b_j\to b\in \partial D$. By continuity, we then obtain
$$
 |z-a| \left( 1 +  \left|  \log \frac { |a-b|  }  {  |z-a|  }     \right| \right) = 1/\lambda'_D(z)
$$
so that the  infimum in (\ref{met2}) is indeed attained as a minimum.

\bigskip

{\bf The infimum for $1/\lambda_D(z)$ and for $1/\lambda''_D(z)$  is a minimum.}

Replacing in the above argument the set $\partial D$ by the closed set ${\mathbb R}^n \setminus D$, we see that also the infimum for $1/\lambda_D(z)$  is attained as a minimum. A similar argument works for $1/\lambda''_D(z)$.

\bigskip

{\bf Continuity of $\lambda'_D$.} 
To consider  continuity of $\lambda'_D$, suppose that $z_j\in D$ and that $z_j\to z\in D$. For each $z_j$, choose points $a_j,b_j\in \partial D$ such that
$$
 |z_j-a_j| \left( 1 +  \left|  \log \frac { |a_j-b_j|  }  {  |z_j-a_j|  }     \right| \right)  = 1/\lambda'_D(z_j) .
$$
Since $z_j \to z\in D$, we see that the points $a_j,b_j$ must lie in a compact subset of $\partial D$, and since $\partial D$ is a closed set, we may pass to subsequences  (of $z_j$ also) without changing notation and assume that $a_j\to a\in\partial D$ and $b_j\to b\in \partial D$. By continuity, we then obtain 
\begin{eqnarray*}
&{}&
\frac { 1 } { \lambda'_D(z_j) }  =  |z_j-a_j| \left( 1 +  \left|  \log \frac { |a_j-b_j|  }  {  |z_j-a_j|  }     \right| \right)
\\ &{}& 
\to  |z-a| \left( 1 +  \left|  \log \frac { |a-b|  }  {  |z-a|  }     \right| \right) = \frac { 1 } { \lambda'_D(z) }  .
\end{eqnarray*}
This proves that $\lambda'_D$ is continuous at $z$ (for if not, then there is a sequence $z_j \to z$ such that $\lambda'_D(z_j)$ has a limit different from $\lambda'_D(z)$, and then every subsequence of $\lambda'_D(z_j)$ also has the same limit, which contradicts the above result). Thus $\lambda'_D$ is continuous.

\bigskip

{\bf Continuity  of $\lambda''_D$ and of $\lambda_D$.} 

Replacing in the above argument the set $\partial D$ by the closed set ${\mathbb R}^n \setminus D$, we see that also $\lambda''_D$ is continuous. A similar argument works for $\lambda_D$. This completes the proof of Lemma~\ref{lem1}.

\section{Proofs of Lemmas~\ref{le5}, \ref{le3}, and \ref{le4}.} \label{pf345}

\subsection{Proof of Lemma~\ref{le3}.} This lemma is obvious, on the basis of the definition of $\lambda_D$. 

\subsection{Proof of Lemma~\ref{le5}.} 

 Let ${\bf e}$ be a unit vector in ${\mathbb R}^n$.  
Define $D= {\mathbb R}^n \setminus (  \{  2 {\bf e}      \}        \cup B^n( {\bf e} ,   \varepsilon ) ) $ where $\varepsilon>0$ is small. Consider $z=0\in D$. 

If we choose $a= {\bf e}$ and $b=2 {\bf e}$ in $ {\mathbb R}^n \setminus D$, we obtain
$$
|z-a| \left( 1 + \left|   \log \frac{  |a-b| } {  | z-a| }     \right| \right) = 1 
$$
so that $1/\lambda_D(z)\leq 1$ and hence $\lambda_D(z)\geq 1$.

To determine $\lambda''_D(0)$, we can only choose $a=(1 -  \varepsilon )  {\bf e} $. If we choose $b\in \partial D$ with $ | b -  {\bf e} | = \varepsilon$, then
$$
\log \frac{ | z-a| } { |a-b| } \geq  \log \frac{ 1-\varepsilon} { 2  \varepsilon } ,
$$
which is large when $\varepsilon$ is small. 
So we choose $b =  2 {\bf e} \in \partial D$, and then
\begin{eqnarray*}
&{}&
\left| \log \frac{ | z-a| } { |a-b| } \right| =  \log \frac{ 1 + \varepsilon} { 1-\varepsilon } 
\\ &{}&
= \log \left(   1 +   \frac{ 2 \varepsilon} { 1-\varepsilon }       \right)
= \log \left(   1 +   2 \varepsilon ( 1 + \varepsilon   +  \varepsilon^2
+   O( \varepsilon^3 ) )   \right)
\\ &{}&
=  2 \varepsilon + 2 \varepsilon^2 -(1/2) (2 \varepsilon)^2 +   O( \varepsilon^3 ) 
\\ &{}&
=  2 \varepsilon + O( \varepsilon^3 ) > 0 
\end{eqnarray*}
so that
\begin{eqnarray*}
&{}&
|z-a| \left( 1 + \left|   \log \frac{  |a-b| } {  | z-a| }     \right| \right) = 
(1 -  \varepsilon )  ( 1 + 2 \varepsilon + O( \varepsilon^3 )  ) 
\\ &{}&
= 1 +  \varepsilon - 2  \varepsilon^2  + O( \varepsilon^3 ) .
\end{eqnarray*}
Hence, for a sufficiently small fixed $ \varepsilon>0$ we have
\begin{equation} \label{ll2}
\frac{1}{ \lambda_D(z) } \leq 1 < 1 + 0.9 \varepsilon \leq   \frac{1}{ \lambda''_D(0) } \leq 1 + 1.1 \varepsilon .
\end{equation}

Then we estimate $ \lambda'_D(z)$. For the infimum in (\ref{met2}) we clearly cannot have $a=2{\bf e}$. 
When $ | a - {\bf e} | = \varepsilon$, write $a={\bf e}+\varepsilon y$ where $|y|=1$ so that
$$
|z-a|^2 = |a|^2 = | {\bf e} + \varepsilon y |^2 = 1 + \varepsilon^2 + 2 \varepsilon {\bf e} \cdot y 
$$
and with $b =  2 {\bf e} $ 
$$
|a-b|^2 = |{\bf e} - \varepsilon y |^2 =  1 + \varepsilon^2 - 2 \varepsilon {\bf e} \cdot y . 
$$
Hence
\begin{eqnarray*}
&{}&
\frac{ 1 } {  |z-a|^2 } = 
\frac{ 1 } { 1 - (- \varepsilon^2 - 2 \varepsilon {\bf e} \cdot y) }
\\ &{}&
 = 
1  + (- \varepsilon^2 - 2 \varepsilon {\bf e} \cdot y) 
+  (- \varepsilon^2 - 2 \varepsilon {\bf e} \cdot y)^2 + O( \varepsilon^3 )
\\ &{}&
= 1  - \varepsilon^2 -  2 \varepsilon {\bf e} \cdot y 
+ 4 \varepsilon^2 (  {\bf e} \cdot y )^2
+ O( \varepsilon^3 )  
\end{eqnarray*}
and so
\begin{eqnarray*}
&{}&
\frac{ |a-b|^2    } {    |z-a|^2    }  = 
(  1 + \varepsilon^2 - 2 \varepsilon {\bf e} \cdot y ) (  1  - \varepsilon^2 -  2 \varepsilon {\bf e} \cdot y 
+ 4 \varepsilon^2 (  {\bf e} \cdot y )^2
+ O( \varepsilon^3 )    ) 
\\ &{}&
= 1 - 4 \varepsilon {\bf e} \cdot y 
+ 8 \varepsilon^2 (  {\bf e} \cdot y )^2
+ O( \varepsilon^3 )   .  
\end{eqnarray*}
Thus
\begin{equation} \label{ll1}
\log \frac{ |a-b|    } {    |z-a|    } 
= - 2 \varepsilon {\bf e} \cdot y 
+ O( \varepsilon^3 ) . 
\end{equation}
Therefore, if $\varepsilon$ is small enough depending on ${\bf e} \cdot y$ and
${\bf e} \cdot y \not= 0$, $\log \frac{ |a-b|    } {    |z-a|    }$ 
has the same sign as $- {\bf e} \cdot y$.  

If $ {\bf e} \cdot y = 0 $, we have $|z-a|=|a-b|$ so that
\begin{eqnarray*}
&{}&
  |z-a| \left( 1 +  \left|  \log \frac{ |a-b|    } {    |z-a|    }       \right|         \right) 
= |z-a| =  \sqrt{ 1 + \varepsilon^2 } =  1 + \frac{1}{2}  \varepsilon^2 + O( \varepsilon^{4} ) .
\end{eqnarray*}
Hence
\begin{equation} \label{ll4} 
\inf \left\{   |z-a| \left( 1 +  \left|  \log \frac{ |a-b|    } {    |z-a|    }       \right|         \right) \colon a,b \in \partial D       \right\} \leq 1  + \frac{3}{4}  \varepsilon^2   
\end{equation}
if $ \varepsilon>0$ is small enough. 
For a sufficiently small fixed $ \varepsilon>0$ this together with (\ref{ll2}) gives
$$
1/\lambda'_D(0) < 1/\lambda''_D(0)  . 
$$

Furthermore,
\begin{eqnarray*}
&{}&
|z-a|  = 1 + \frac{1}{2} ( \varepsilon^2 + 2 \varepsilon {\bf e} \cdot y )
-  \frac{1}{8} ( \varepsilon^2 + 2 \varepsilon {\bf e} \cdot y )^2 
+ O( \varepsilon^3 )  
\\ &{}&
= 1 + \varepsilon {\bf e} \cdot y + \frac{1}{2}  \varepsilon^2 
\left( 1 -  ( {\bf e} \cdot y )^2 \right) 
+ O( \varepsilon^3 ) . 
\end{eqnarray*}
So if $|  {\bf e} \cdot y | \leq \varepsilon^{3/2}$, then
$$
|z-a|  \geq 1 + \frac{1}{2}  \varepsilon^2 + O( \varepsilon^{5/2} )
$$
so that
$$
 |z-a| \left( 1 +  \left|  \log \frac{ |a-b|    } {    |z-a|    }       \right|         \right) 
 \geq 1 + \frac{1}{2}  \varepsilon^2 + O( \varepsilon^{5/2} )
 \geq  1 + \frac{1}{4}  \varepsilon^2   
$$
if $ \varepsilon>0$ is small enough. 

To complete the proof of Lemma~\ref{le5} it remains to show that
$$
 |z-a| \left( 1 +  \left|  \log \frac{ |a-b|    } {    |z-a|    }       \right|         \right) 
$$
is greater than $1$ by some fixed amount when $| {\bf e} \cdot y | > \varepsilon^{3/2}$ and $ \varepsilon>0$ is small enough.
By (\ref{ll1}), the sign is $\log \frac{ |a-b|    } {    |z-a|    } $ is the same as the sign of $- {\bf e} \cdot y$ when $|  {\bf e} \cdot y | > \varepsilon^{3/2}$ and $ \varepsilon>0$ is small enough.

Suppose that $|a-b|  >  |z-a|$ so that $ {\bf e} \cdot y <0$. Then
\begin{eqnarray*}
&{}&
 |z-a| \left( 1 +  \left|  \log \frac{ |a-b|    } {    |z-a|    }       \right|         \right) 
\\ &{}&
  = \left( 1 + \varepsilon {\bf e} \cdot y + \frac{1}{2}  \varepsilon^2 
\left( 1 -  ( {\bf e} \cdot y )^2 \right) 
+ O( \varepsilon^3 ) \right) \times  
\\ &{}&
\times \left(  1  - 2 \varepsilon {\bf e} \cdot y 
+ O( \varepsilon^3 )  \right) 
\\ &{}&
= 1   -  \varepsilon {\bf e} \cdot y +  \frac{1}{2}  \varepsilon^2 (1 -  5( {\bf e} \cdot y)^2 ) 
+ O( \varepsilon^3 ) .  
\end{eqnarray*}
If $-1\leq {\bf e} \cdot y \leq -1/3$, this gives
$$
 |z-a| \left( 1 +  \left|  \log \frac{ |a-b|    } {    |z-a|    }       \right|         \right) 
 \geq 1 + \varepsilon/3 - 2 \varepsilon^2 + O( \varepsilon^3 ) 
 > 1 + \varepsilon/4 
$$
if $ \varepsilon>0$ is small enough. 
If $-1/3 < {\bf e} \cdot y < 0$, this gives
$$
 |z-a| \left( 1 +  \left|  \log \frac{ |a-b|    } {    |z-a|    }       \right|         \right) 
 \geq 1 + (2/9) \varepsilon^2  + O( \varepsilon^3 )
 > 1 + (1/9) \varepsilon^2  
$$
if $ \varepsilon>0$ is small enough.

Suppose that $|a-b| <  |z-a|$ so that $ {\bf e} \cdot y >0$. Then
\begin{eqnarray*}
&{}&
 |z-a| \left( 1 +  \left|  \log \frac{ |a-b|    } {    |z-a|    }       \right|         \right) 
\\ &{}&
  = \left( 1 + \varepsilon {\bf e} \cdot y + \frac{1}{2}  \varepsilon^2 
\left( 1 -  ( {\bf e} \cdot y )^2 \right) 
+ O( \varepsilon^3 ) \right) \times  
\\ &{}&
\times \left(  1  + 2 \varepsilon {\bf e} \cdot y 
+ O( \varepsilon^3 )  \right) 
\\ &{}&
= 1   + 3  \varepsilon {\bf e} \cdot y +  \frac{1}{2}  \varepsilon^2 (1 + 3 ( {\bf e} \cdot y)^2 ) 
+ O( \varepsilon^3 )
> 1 +  \frac{1}{3}  \varepsilon^2    
\end{eqnarray*}
if $ \varepsilon>0$ is small enough.

Thus
\begin{equation} \label{ll5} 
\inf \left\{   |z-a| \left( 1 +  \left|  \log \frac{ |a-b|    } {    |z-a|    }       \right|         \right) \colon a,b \in \partial D       \right\} \geq 1  + \frac{1}{9}  \varepsilon^2    > 1
\end{equation}
when $\varepsilon$ is small enough, 
as required. Combining these results we obtain the claim of Lemma~\ref{le5}.

\subsection{Proof of Lemma~\ref{le4}.} 
Let ${\bf e}$ be a unit vector in ${\mathbb R}^n$.  Define
$D_2 = {\mathbb R}^n \setminus   \{  {\bf e} , 2 {\bf e}      \}$. 
Define $D_1$ to be the same as the domain $D$ in the proof of Lemma~\ref{le5}, that is, $D_1= {\mathbb R}^n \setminus (  \{  2 {\bf e}      \}        \cup B^n( {\bf e} ,   \varepsilon ) ) $ where $\varepsilon>0$ is small. Consider $z=0\in D_1 \subset D_2$. 

To determine $\lambda_{D_2}(0) $, $\lambda'_{D_2}(0) $, and $\lambda''_{D_2}(0) $, we clearly need to take $ a= {\bf e} $ and $b= 2 {\bf e}$ in each case. Now $|z-a| = |a| = |a-b|=1$ in each case so that
$$
\lambda_{D_2}(0) = \lambda'_{D_2}(0) = \lambda''_{D_2}(0) = 1 .
$$
By (\ref{ll2}), (\ref{ll4}),  and (\ref{ll5}) in the proof of Lemma~\ref{le5}, we see that when $\varepsilon>0$ is small enough, we have
$$
 \frac{1}{ 1 + 1.1 \varepsilon }  \leq    \lambda''_{D_1}(0)  \leq   \frac{1}{ 1 + 0.9 \varepsilon } <   \frac{1}{ 1 + (3/4) \varepsilon^2 }  \leq    \lambda'_{D_1}(0)  \leq   \frac{1}{ 1 + (1/9) \varepsilon^2 }    .
$$
This implies (\ref{comp4}) for $z=0$. The proof of Lemma~\ref{le4} is complete.

\section{Proof of Lemma~\ref{le7}.} \label{pf7}

It will be useful to prove Lemma~\ref{le7} before proving Lemma~\ref{lem2}. Let the assumptions of Lemma~\ref{le7} be satisfied and let $(a',b')$ be any pair for which the infimum in (\ref{met1}) is attained. 

Let us first note, by giving a sketch of the proof,  that if $a'$ is given, then $b'\in \partial D$ unless, possibly, if $ |a'-b'|   = |z-a'| $. For otherwise $b'$ lies in the open set ${\mathbb R}^n \setminus \overline{D}$, in which case we may vary $b'$ in a small ball, causing $\left| \log \frac{  |a'-b'|    } {  |z-a'|   }  \right|$, which is $>0$ to begin with, to get smaller (note that $a'$ is fixed in this argument). 

Similarly, if $b'$ is given, then $a'\in \partial D$ unless, possibly, if $ |a'-b'|   = |z-a'| $. For otherwise $a'$ lies in the open set ${\mathbb R}^n \setminus \overline{D}$, in which case we may vary $a'$ in a small ball (or, in one case, interchange the roles of $a'$ and $b'$), causing $|z-a'| \left(  1 + \left| \log \frac{  |a'-b'|    } {  |z-a'|   }  \right|  \right)  $ to get smaller (note that $b'$ is fixed in this argument). 

We now present the details of the argument in this latter case when $z$ and $b'$ are fixed and $a'$ is treated as a variable. Let $\alpha\in {\mathbb R}^n$ have a small norm and set $a''=a'+\alpha$. Then, ignoring in all formulas below all terms that are $O(|\alpha|^2)$ as $\alpha\to 0$, we have
$$
|z-a''|^2 = |z-a'-\alpha|^2 = |z-a'|^2 - 2(z-a')\cdot \alpha ,
$$
$$
|z-a''| =   |z-a'| - \frac{ (z-a')\cdot \alpha } {  |z-a'| } 
=   |z-a'|\left( 1 - \frac{ (z-a')\cdot \alpha } {  |z-a'|^2 }\right)  ,
$$
$$
|b'-a''| =  |b'-a'| - \frac{ (b'-a')\cdot \alpha  } {  |b'-a'| }
=   |b'-a'|\left( 1 - \frac{ (b'-a')\cdot \alpha } {  |b'-a'|^2 }\right) ,
$$
$$
\log \frac{  |a''-b'|    } {  |z-a''|   } 
= \log \frac{  |a'-b'|    } {  |z-a'|   } 
- \frac{ (b'-a')\cdot \alpha } {  |b'-a'|^2 } + \frac{ (z-a')\cdot \alpha } {  |z-a'|^2 } .
$$

Suppose first that $ |a'-b'|   \not= |z-a'| $, and let $s$ denote the sign of
$ \log \frac{  |a'-b'|    } {  |z-a'|   }$. Then
$$
\left|  \log \frac{  |a''-b'|    } {  |z-a''|   }  \right| 
= \left|   \log \frac{  |a'-b'|    } {  |z-a'|   }  \right| 
+s \left(   \frac{ (z-a')\cdot \alpha } {  |z-a'|^2 } - \frac{ (b'-a')\cdot \alpha } {  |b'-a'|^2 }  \right) .
$$
A calculation shows that if we denote
$$
 |z-a''| \left(  1 + \left| \log \frac{  |a''-b'|    } {  |z-a''|   }  \right|  \right)
 -
  |z-a'| \left(  1 + \left| \log \frac{  |a'-b'|    } {  |z-a'|   }  \right|  \right)
$$
by $\Delta$, then
\begin{equation} \label{del1}
\Delta  
=
- \frac{ (z-a')\cdot \alpha } {  |z-a'| } 
\left( 1 +   \left| \log \frac{  |a'-b'|    } {  |z-a'|   }  \right|  \right)
\end{equation}
\begin{equation*}
+ |z-a'| s \left(   \frac{ (z-a')\cdot \alpha } {  |z-a'|^2 } - \frac{ (b'-a')\cdot \alpha } {  |b'-a'|^2 }  \right) 
= \alpha \cdot A ,
\end{equation*}
say. Since $\alpha$ can vary in a full small neighborhood of the origin in ${\mathbb R}^n$, it is possible to make $\alpha \cdot A$ negative by a suitable choice of $\alpha$ provided that $A\not= 0$. This shows that the situation considered could not have been extremal. 

Next consider the possibility that $A=0$. We can write $A=a_1 (b'-a') +a_2 (z-a')$ for certain real numbers $a_1$ and $a_2$, where $a_1= -s|z-a'|/ |b'-a'|^2\not= 0$. Hence we can have $A=0$ only if $z-a'$ is a real multiple of $b'-a'$, which means that $a'$, $b'$, and $z$ lie on the same straight line. We have 
$|z-a'| a_2= s-1 -   \left| \log \frac{  |a'-b'|    } {  |z-a'|   }  \right| <0$ and $sa_1<0$, so that if $s=1$, then $z-a'$ is a negative multiple of $b'-a'$, while if $s=-1$, then  $z-a'$ is a positive multiple of $b'-a'$. Thus, if $ |a'-b'|   > |z-a'| $, then $a'$ is between $z$ and $b'$, while if $ |a'-b'|   < |z-a'| $, then $b'$ is between $z$ and $a'$.  

Consider the case $ |a'-b'|   > |z-a'| $, so that $s=1$ and hence
$$
|b'-a'|^{-2} (b'-a') = -|z-a'|^{-2} \log \frac{  |a'-b'|    } {  |z-a'|   } (z-a'),
$$
which, after taking norms of both sides, implies that
$$
 \log \frac{  |a'-b'|    } {  |z-a'|   } = \frac{  |z-a'|      } {  |a'-b'|   }.
$$ 
Thus the  number $\gamma=\frac{  |a'-b'|    } {  |z-a'|   } >1$ must be a solution of the equation $\log x=1/x$, which equation, for $x>0$, has the unique solution $x\approx 1.76322$. 

Let $a'$ vary on the line segment between $z$ and $(z+b')/2$. 
Write $R=|z-b'|$ and $0<u=|z-a'|<R/2$ since $|z-a'|<|a'-b'|$.
Then
$$
\varphi(u) := u \left( 1 + \log \frac{ R-u}{u} \right) =
 |z-a'| \left(  1 + \left| \log \frac{  |a'-b'|    } {  |z-a'|   }  \right|  \right)
$$
satisfies
$$
\varphi'(u) = 1 + \log \frac{ R-u}{u} -  \frac{u}{R-u} -1
=\log \frac{ R-u}{u} -  \frac{u}{R-u}
$$
and
$$
\varphi''(u) = -  \frac{1}{R-u} -\frac{1}{u} - \frac{ R-u +u } {  (R-u)^2 } <0 .
$$
Incidentally, we see again that $\varphi'(u) =0$ exactly when $\gamma= |a'-b'|/|z-a'| = (R-u)/u$ satisfies $\log \gamma  = 1/\gamma$. 
Since $\varphi''(u)<0$, a local extremum is a local maximum, so we do not find the infimum when $\gamma  \approx 1.76322$. This shows that when $|z-a'|<|a'-b'|$, if $b'$ is given, then the infimum on the right  hand side of (\ref{met1}) can only be attained when $a'\in \partial D$. 

Then consider the case $ |a'-b'|   < |z-a'| $, so that $s=-1$ and hence
$$
|z-a'|^{-2}  (2 -    \log \frac{  |a'-b'|    } {  |z-a'|   }  ) (z-a') = |b'-a'|^{-2} (b'-a') .
$$
Taking norms gives for the quantity $\gamma =  |a'-b'|   / |z-a'| <1$ that
$$
\gamma (2-\log\gamma)=1.
$$
The equation $x(2-\log x)=1$ has the unique solution $x\approx 0.317844$ on the interval $(0,1)$. 

Recall that now $b'$ is between $z$ and $a'$. Write $R=|z-b'|$ and $u=|a'-b'|$ so that $|z-a'|=R+u$. We  ask what happens if we interchange the roles of $a'$ and $b'$, as we may when looking for the infimum on the right hand side of  (\ref{met1}). 
Then
$$
 |z-a'| \left(  1 + \left| \log \frac{  |a'-b'|    } {  |z-a'|   }  \right|  \right)
 =(R+u) \left( 1 +  \left| \log \frac{  u}{R+u}  \right|   \right)
$$
is replaced by 
$$
 |z-b'| \left(  1 + \left| \log \frac{  |a'-b'|    } {  |z-b'|   }  \right|  \right)
 = R \left(  1 + \left| \log \frac{ u    } {  R   }  \right|  \right)  .
$$ 
We ask whether we get a smaller number in this way, that is, whether
$$
R \left(  1 + \left| \log \frac{ u    } {  R   }  \right|  \right) < 
(R+u)  \left(  1 + \left| \log \gamma  \right|  \right) .
$$
This is equivalent to
$$
1 + | \log (1/\gamma - 1) | < 
( 1 - \gamma )^{-1} 
 ( 1 + | \log \gamma | )  ,
$$
which a calculation shows to be valid, using the value 
$\gamma= 0.317844$. This shows that we do not get the infimum on the right hand side of (\ref{met1}) for the original $a'$ and $b'$. Thus, if $|z-a'|>|a'-b'|$ and $b'$ is given, we have $a' \in\partial D$ when the infimum is reached.

Combining the results above, we see that whenever $a',b'\in {\mathbb R}^n \setminus D$ are such that the infimum on the right hand side of (\ref{met1}) is reached (recognizing that such pairs $a',b'$ need not be unique), then $a', b' \in\partial D$ except possibly if $|z-a'| = |a'-b'|$. For if we assume that we know $a'$, we have seen that $b'\in \partial D$, and if we then assume that we know this $b'\in \partial D$, we have seen that $a'\in \partial D$. 

Suppose now that in the extremal situation we have  $|z-a'| = |a'-b'|$ and that $b'$ lies in the open set ${\mathbb R}^n \setminus \overline{D}$. Then we can follow an arc of a circle contained in an $(n-1)-$dimensional sphere, say $S$, centered at $a'$ and containing $b'$ until we reach a point $b''$ such that $| a' - b'' | = |a'-b'|$ (by definition) and $b''\in \partial D$ (if this is the case then indeed there would be extremal pairs $a',b'$ such that not both $a',b'$ lie on $\partial D$). For if there is no such point $b''$ then $S\subset  {\mathbb R}^n \setminus \overline{D}$, and hence a neighborhood of $S$, and therefore also a neighborhood of $b'$, lies in the open set ${\mathbb R}^n \setminus \overline{D}$. Thus, as seen before, we can move $b'$ slightly to make the quantity $ |z-a'| \left(  1 + \left| \log \frac{  |a'-b'|    } {  |z-a'|   }  \right|  \right) $ smaller, which shows that we were not in an extremal situation to begin with. It follows that when $|z-a'| = |a'-b'|$, it is possible to choose $b'$ so that $b'\in \partial D$. 

Suppose now that we are in an extremal situation with $|z-a'| = |a'-b'|$ and that $b'\in \partial D$. We keep $b'$ fixed. If $a' \in {\mathbb R}^n \setminus \overline{D}$, we   follow an arc of a circle contained in an $(n-1)-$dimensional sphere, say $S'$, centered at $z$ and containing $a'$ until we reach a point $a''$ such that $| a' - z | = |a''-z|$ (by definition) and $a''\in \partial D$. If this is not possible, then, as above, a neighborhood of $a'$ lies in the open set ${\mathbb R}^n \setminus \overline{D}$, and we can move $a'$ slightly to make the quantity $ |z-a'| \left(  1 + \left| \log \frac{  |a'-b'|    } {  |z-a'|   }  \right|  \right) $ smaller, so we were not in an extremal situation to begin with. The details of this argument are as follows. We look for a point $a''=a'+\alpha$, say, where $\alpha\in {\mathbb R}^n$ and $|\alpha|$ is small, such that
$$
 |z-a''| \left(  1 + \left| \log \frac{  |a''-b'|    } {  |z-a''|   }  \right|  \right) < |z-a'| .
$$
We ignore all terms that are $O(|\alpha|^2)$ as $\alpha \to 0$. 
Using the same notation as just before (\ref{del1}), now that $ \left| \log \frac{  |a'-b'|    } {  |z-a'|   }  \right| =0$,  (\ref{del1})  is replaced by
\begin{eqnarray} \label{del2}
\Delta &=& - \frac{ (z-a')\cdot \alpha } {  |z-a'| } 
+ |z-a'|  \left| \left(   \frac{ (z-a')\cdot \alpha } {  |z-a'|^2 } - \frac{ (b'-a')\cdot \alpha } {  |b'-a'|^2 }  \right) \right| 
\notag \\
&=& \alpha \cdot A_1 + |  \alpha \cdot A_2 |  ,
\end{eqnarray}
say.  Clearly $A_1\not= 0$. If $A_2$ is not a real multiple of $A_1$, we can choose $\alpha\not= 0$ so that $ \alpha \cdot A_2=0$ while $\alpha \cdot A_1 \not= 0$. Then, replacing $\alpha$ by $-\alpha$, if necessary, we can achieve $\Delta<0$, which shows that we do not have an infimum at the pair $a',b'$. Suppose then that $A_2$ is  a real multiple of $A_1$. Then $b'-a'$ is a real multiple of $z-a'$, so that $a'$, $b'$, and $z$ lie on the same straight line, and since  $|z-a'| = |a'-b'|$, we have $a'=(z+b')/2$. This must then be the situation for every pair $a',b'$ that minimizes
$$
 \left\{  |z-a| \left(  1 + \left| \log \frac{  |a-b|    } {  |z-a|   }  \right|  \right)  : a,b\in {\mathbb R}^n \setminus D      \right\}  
 $$
if there is more than one such pair.

This situation can occur as we have seen in the proof of Lemma~\ref{le5}. Let ${\bf e}$ be a unit vector in ${\mathbb R}^n$, where $n\geq 2$, and set $D= {\mathbb R}^n \setminus (\{2{\bf e} \} \cup \overline{B}({\bf e},\varepsilon))$, where $\varepsilon>0$ is very small. Suppose that $z=0\in D$. The above analysis can be used to show that to minimize the right hand side of (\ref{met1}), we must choose $b=2{\bf e} \in \partial D$ and $a={\bf e}$.  Here $a\in {\mathbb R}^n \setminus \overline{D}$, and we cannot minimize the  right hand side of (\ref{met1}) by using any $a\in \partial D$. 

Thus we have seen that whenever $a',b'\in {\mathbb R}^n \setminus D$ are such that the infimum on the right hand side of (\ref{met1}) is reached (recognizing that such pairs $a',b'$ need not be unique), then $a', b' \in\partial D$ except possibly if $|z-a'| = |a'-b'|$. Furthermore, we have seen that if $|z-a'| = |a'-b'|$, then $b'\in \partial D$, and also we may choose $a'\in \partial D$ (even though there may be possible choices for $a'$ outside $\overline{D}$ also) except possibly when $a'=(z+b')/2$ for every minimizing pair $a',b'$ (and of course the distance $|z-a'|$ must be the same for all such $a'$, even though this distance is then greater than $d(z,\partial D)$). We have seen that this last exceptional case may occur. When we are not in this exceptional case, we have $\lambda_D(z) = \lambda'_D(z)$, while in the exceptional case we have
$\lambda_D(z) > \lambda'_D(z)$. 

This proves Lemma~\ref{le7}.

\section{Proof of Lemma~\ref{lem2}.} \label{pf2}

Let the assumptions of Lemma~\ref{lem2}  be satisfied and suppose that $z\in D$. 
It is immediate from the definitions that 
$ \lambda''_D(z) \leq \lambda'_D(z) \leq \lambda_D(z)$. Thus we only need to prove that $\lambda_D(z) \leq C_0 \lambda''_D(z)$ for a positive absolute constant $C_0$.

We will need the following lemma.

\begin{lemma} \label{le1}
Suppose that $T>0$ and define 
\begin{equation} \label{ft0}
f(y) = y \left(  1 + \left|  \log \frac{T}{y}   \right|   \right) 
\end{equation}
for $y>0$. Then $f$ is a strictly increasing function of $y$ for $y>0$.
\end{lemma}

\medskip

{\bf Proof of Lemma~\ref{le1}. } 
For $0<y<T$, we have
$$
f(y) = y \left(  1 +  \log \frac{T}{y}    \right) 
$$
so that
$$
f'(y) = 1 +  \log \frac{T}{y}  -1 = \log \frac{T}{y} >0 .
$$
For $y>T$, we have
$$
f(y) = y \left(  1 -  \log \frac{T}{y}    \right) 
$$
so that
$$
f'(y) = 1 -  \log \frac{T}{y}  +1 = 2 + \log \frac{y}{ T } >0 .
$$
Since $f$ is continuous (in particular, at $y=T$), this proves Lemma~\ref{le1}.
$\Box$

We return to the proof of Lemma~\ref{lem2}. For a fixed $z\in D$, the inequality
$\lambda_D(z) \leq C_0 \lambda''_D(z)$ is equivalent to
\begin{eqnarray} \label{ine1}
&{}&
\inf\left\{  |z-a| \left(  1 + \left| \log \frac{  |a-b|    } {  |z-a|   }  \right|  \right)  : a,b\in \partial D, \,\, |z-a|=  d(z,\partial D)      \right\}   \notag
\\    & \leq & C_0 
\inf\left\{  |z-a| \left(  1 + \left| \log \frac{  |a-b|    } {  |z-a|   }  \right|  \right)  : a,b\in {\mathbb R}^n \setminus D      \right\}  .
\end{eqnarray} 

Choose $a,b\in \partial D$ such that $|z-a| =  d(z,\partial D) =: d>0$ and such that 
$$
\beta := \left|  \log \frac { |a-b|  }  {  d  }     \right|   \geq 0 
$$
is minimal (over all $a,b\in \partial D$ with $|z-a| =  d(z,\partial D)$), so that
\begin{eqnarray} \label{ine2}
&{}&
\inf\left\{  |z-a| \left(  1 + \left| \log \frac{  |a-b|    } {  |z-a|   }  \right|  \right)  : a,b\in \partial D, \,\, |z-a|=  d(z,\partial D)      \right\}   \notag
\\    & = &  d (1+\beta) .
\end{eqnarray} 

Next, choose $a',b'\in {\mathbb R}^n \setminus D$ such that with the notation
$|z-a'| =   d'\geq d$ and
$$
\beta' = \left|  \log \frac { |a'-b'|  }  {  d'  }     \right|   \geq 0 
$$
we have
\begin{eqnarray} \label{ine3}
&{}&
\inf  \left\{  |z-a| \left(  1 + \left| \log \frac{  |a-b|    } {  |z-a|   }  \right|  \right)  : a,b\in {\mathbb R}^n \setminus D      \right\}    \notag
\\    & = &  d' (1+\beta') .
\end{eqnarray} 

We need to prove that
\begin{equation} \label{ine4}
d (1+\beta) \leq C_0 d' (1+\beta') .
\end{equation} 

Suppose that it is known how $a'$ and $b'$ may be chosen, noting that the choice need not be unique, so we are considering any fixed possible choice. 
By Lemma~\ref{le7}, whenever $a',b'\in {\mathbb R}^n \setminus D$ are such that the infimum on the left hand side of (\ref{ine3}) is reached, $a', b' \in\partial D$ except possibly if $|z-a'| = |a'-b'|$. 
Further, if $|z-a'| = |a'-b'|$, then $b'\in \partial D$, and in addition, we may choose $a'\in \partial D$ (even though there may be possible choices for $a'$ outside $\overline{D}$ also) except possibly when $a'=(z+b')/2$ for every minimizing pair $a',b'$ (and  the distance $|z-a'|$ must be the same for all such $a'$, even though this distance is then greater than $d(z,\partial D)$). We have seen that this last exceptional case may occur. 

Suppose first that $d'=d$. Then we necessarily have $a'\in \partial D$ since then $|z-a'|=d'=d=d(z,\partial D)$.  Since $b'\in \partial D$, as we have seen, it follows that the two infima in (\ref{ine1})  are equal. Hence   (\ref{ine4})  holds with $C_0=1$. 

Suppose then that $d'>d$, hence $a'\not= a$. If $\beta=0$, then  (\ref{ine1}) holds with $C_0=1$. Thus we assume from now on that $\beta>0$. Note that $d$ is the same, no matter which possible point $a$ is chosen in case there is a choice for the minimum. 

Recall the notation $A(z,r_1,r_2) = \{ w\in {\mathbb R}^n \colon r_1<|w-z|<r_2    \}$. 
Since $\beta>0$, we have
\begin{equation}  \label{a1}
A(a,  e^{-\beta} d,  e^{\beta} d ) \cap \partial D  =\emptyset  .
\end{equation}
Since $z\in A(a,  e^{-\beta} d,  e^{\beta} d )$, it follows that
\begin{equation}  \label{a11}
A(a,  e^{-\beta} d,  e^{\beta} d ) \subset   D  .
\end{equation}
Also $B^n(z,d)\subset D$, and hence
\begin{equation}  \label{a2}
B^n(z,d) \cup A(a,  e^{-\beta} d,  e^{\beta} d ) \subset   D  .
\end{equation}

Now (\ref{ine1}) holds with $C_0=1$ if $d'\geq d (1 + \beta)$. Hence we will assume that
\begin{equation}  \label{d2}
d' < d (1 + \beta) .
\end{equation}

Recall the number $t_0\approx 
1.14619$
given by (\ref{t00}) and its properties.

We next consider two cases.

{\bf The case $\beta\geq t_0$.}  
If
\begin{equation}  \label{beta1-t0}
\beta\geq t_0 
\end{equation}
then
$$
e^{\beta}d -d \geq d(1+\beta) .
$$
Thus
$$
|z-a'| =d' < d(1+\beta) \leq e^{\beta}d -d
$$
so that
$$
|a-a'| \leq |z-a|+|z-a'|=d+d' < e^{\beta}d .
$$
Since $a'\notin A(a,  e^{-\beta} d,  e^{\beta} d )$ by (\ref{a2}), it follows that 
\begin{equation}  \label{a2small}
a'\in \overline{B}^n(a, e^{-\beta} d )\setminus  B^n(z,d)  .
\end{equation}

Next, still by (\ref{a2}),  because $b'\in \partial D$, we have either 
\begin{equation}  \label{b21}
b' \in \overline{B}^n(a, e^{-\beta} d )\setminus  B^n(z,d) 
\end{equation}
or 
\begin{equation}  \label{b22}
|b'-a|\geq e^{\beta}d .
\end{equation}

Suppose  that  (\ref{b21}) holds.  Note that $e^{\beta}>2$. Then 
$$
|a'-b'| \leq 2 e^{-\beta} d <  d \leq  d'
$$
so that
$$
\beta' = \log \frac { d' } { |b'-a'|  } \geq \log \frac{ d } { 2 e^{-\beta} d } 
= \beta - \log 2 .
$$
Hence
\begin{eqnarray} \label{d1d2t0}
\frac{ d (1+\beta) } { d' (1+\beta') }  & \leq &
 \frac{  1+\beta } { 1+\beta' } \leq \frac{  1+\beta } {  1+\beta - \log 2 }
 \\  
 & \leq &
\frac{  1+t_0 } {  1+ t_0 - \log 2 } \approx 1.47703  
<  1.48 . \notag
\end{eqnarray}
This implies (\ref{ine1}) with $C_0=1.48$.

We continue with the discussion of the case when (\ref{beta1-t0}) and (\ref{a2small}) hold. 
Suppose next that (\ref{b22}) holds. Then
\begin{eqnarray} \label{abd} 
|b'-a'|  & \geq &  |b'-a| - |a'-a| \geq e^{\beta}d - e^{-\beta}d
\\
& > &  ( e^{\beta} - 1) d \geq (1+\beta) d > d' . \notag
\end{eqnarray} 
Hence
$$
\beta' = \log \frac { |b'-a'| } { d'} .
$$

Consider $y=|b'-a'|>d'$ as a fixed number, and consider the function
\begin{equation} \label{f-inc}
f(t)=t (1+\log (y/t)) 
\end{equation}
for $d\leq t\leq d(1+\beta)$. This range contains the value $t=d'$, and we have $t<y$ for all these $t$. By Lemma~\ref{le1}, $f$ is strictly 
 increasing, and so $f(t)\geq f(d)
=d(1+\log (y/d))$ for all these $t$. In particular, 
$$
d' ( 1 + \beta') =f(d') \geq d(1+\log (y/d))  .
$$
Hence, (\ref{ine1}) holds with a certain $C_0\geq 1$ provided that
$$
C_0 (1+\log (y/d)) \geq 1+\beta ,
$$
which is equivalent to 
$$
C_0 \log (|b'-a'|/d) \geq \beta + 1 - C_0 ,
$$
that is,
\begin{equation} \label{b2a2}
|b'-a'| \geq d \exp\left( \frac {\beta + 1} {C_0}  - 1 \right)   .
\end{equation}
From (\ref{abd}), we see that $|b'-a'| \geq ( e^{\beta} - 1) d$, so that
(\ref{b2a2}) holds if 
$$
e^{\beta} - 1 \geq \exp\left( \frac {\beta + 1} {C_0}  - 1 \right)  ,
$$
that is,
$$
\frac {  \log ( e^{\beta} - 1  ) + 1 } { \beta + 1 }   \geq \frac{1}{C_0}  .
$$
Now $\beta\geq t_0$. For $t\geq t_0$, define
$$
f(t) =  \frac {  \log ( e^{t} - 1  ) + 1 } { t + 1 }   .
$$
Then
$$
f'(t) (e^t - 1) (t+1)^2 = t e^t + 1 - (e^t - 1) \log (e^t - 1) >0
$$
since $e^t - 1 <e^t$ and $ \log (e^t - 1) <\log e^t =t$. Thus for all $t\geq t_0$, we have (recall that $e^{t_0} = 2+t_0$)
$$
f(t)\geq f(t_0) =  \frac {  \log ( t_0 +1  ) + 1 } { t_0 + 1 } \approx 0.821778608 
$$
so that (\ref{ine4})  holds with
$$
C_0=1/f(t_0) \approx 1.21687
< 1.22  .
$$

This concludes the proof when (\ref{beta1-t0}) holds. We have now completed the case $\beta \geq t_0$. 

\bigskip

{\bf The case $0<\beta< t_0$.}  
Suppose then that 
\begin{equation}  \label{b23}
0< \beta < t_0   
\end{equation}
so that $e^{\beta} -1<1+\beta$. 
Then with $C_0=1+t_0<2.15$, we have (since $d\leq d'$)
$$
d(1+\beta)\leq C_0 d' \leq C_0 d' (1+\beta') 
$$
so that (\ref{ine1}) holds with $C_0=1+t_0$.

We have now completed the proof of (\ref{ine4}) and seen that the choice $C_0 =2.15$ works in all cases.

This completes the proof of Lemma~\ref{lem2}.

\section{Proof of Theorem~\ref{th-qc1}.} \label{s5}

Let the assumptions of Theorem~\ref{th-qc1} be satisfied and suppose first that $f$ is a $K-$quasiconformal homeomorphism of $D$ onto $D'$. By Lemma~\ref{lem2} we may use $\lambda''$ instead of $\lambda$.  We may perform the following normalizations. 

We fix $z\in D$ and assume that $z=f(z)=0$ and  $d(0,\partial D)=d(0,\partial D')=1$. Choose $a\in \partial D$ and $a'\in \partial D'$ so that $|a|=|a'|=1$. Write
$$
\beta = \min  \left\{     \left| \log   |a''-b|    \right|    : a'',b\in \partial  D , |a''| = 1    \right\} ,
$$
$$
\beta' = \min   \left\{   \left| \log   |a''-b|    \right|   : a'',b\in \partial  D' , |a''| = 1    \right\} ,
$$
so that
$$
 \frac{ 1}{\lambda''_D(z) } = 1 + \beta ,
$$
$$
 \frac{ 1}{\lambda''_{D'}(f(z)) } = 1 + \beta' .
$$
We need to prove that
\begin{equation} \label{b111} 
1+\beta \leq C_1 (1+\beta')
\end{equation}
for some constant $C_1>1$ depending on $n$ and $K$ only. Since $f^{-1}$ is $K-$quasiconformal, 
the same argument applied to $f^{-1}$ instead of $f$ then implies that 
$  1+\beta'  \leq C_1 (1+\beta)  $. 
This then proves 
Theorem~\ref{th-qc1}. Thus we proceed to prove (\ref{b111}).  

Define
\begin{equation} \label{c2}
C_2 =  K_I(f)^{1/(n - 1)} \log (80 \lambda_n^4)         
\end{equation}
where $\lambda_n \geq 4$ is a positive constant depending on $n$ only, to be mentioned in connection with (\ref{ll6}) below. 
If $\beta \leq C_2$, then (\ref{b111}) holds with $C_1=1+C_2$. Hence we may assume that $\beta >  C_2$.  

Recall the definition 
$
A(a,r_1,r_2) = \{ x\in {\mathbb R}^n : r_1< |x-a| < r_2 \} .
$

The definition of $\beta$ implies that 
$z=0 \in A(a,e^{-\beta}, e^{\beta}) \subset D$. Therefore the topological ring domain
$f(  A(a,e^{-\beta}, e^{\beta})     )$ containing $f(0)=0$ is a subset of $D'$. 

Choose $\beta_0>0$ so that $\beta_0<\beta$ but $\beta_0$ is very close to $\beta$. Define $A_0= A(a,e^{-\beta_0}, e^{\beta_0})$ so that the closure
$\overline{A_0}$ of $A_0$ is a compact subset of $D$. Thus $A_0' = f( \overline{A_0}   )$ is a compact subset of $D'$, and the open set ${\mathbb R}^n \setminus A_0'$ has exactly two components, an unbounded component $U_0$ and a bounded component $U_1$. Let us include the point at infinity in $U_0$ for the arguments that follow. Then $U_0$ is an open subset of $\overline{ {\mathbb R}^n  }$. 

For the definition and properties of the modulus $M(\Gamma)$ of a path family $\Gamma$ we refer to \cite{Vuo88}, Chapter~II.

Let $\Gamma$ be the family of all rectifiable paths in $A_0$ joining the two boundary components of $A_0$. Then the modulus $M(\Gamma)$ is given by (\cite{Vuo88}, p.~129) 
$$
M(\Gamma) = \frac{ \omega_{n-1} } { \left( \log \frac{e^{\beta_0} }{e^{-\beta_0}}  \right)^{n-1} } =   \frac{ \omega_{n-1} } { ( 2\beta_0 )^{n-1} } ,
$$
where $\omega_{n-1}$ is the $(n-1)-$dimensional measure of the unit sphere $S^{n-1}(0,1)$. Then by \cite{Vuo88}, p.~53, we have, writing $K_I=K_I(f)$, 
$$
M( f( \Gamma ) ) \leq K_I M(\Gamma) =    \frac{ K_I  \omega_{n-1} } { ( 2\beta_0 )^{n-1} } .  
$$

Let $x_0$ be an arbitrary point of $U_1$. Given $x_0$, define
$$
r_1 = \max \{ |y-x_0| \colon y \in \overline{U_1} \} 
$$
and
$$
r_2 = \min \{ |y-x_0| \colon y \in \overline{U_0} \} . 
$$
Note that we may have $r_2\leq r_1$. Choose a point $x_1\in  \overline{U_1}$ with $|x_1-x_0|=r_1$ and a point $x_2\in \overline{U_0}$ with $|x_2-x_0|=r_2$. Then $x_0,x_1\in  \overline{U_1}$ and $x_2, \infty\in  \overline{U_0}$. Now \cite{Vuo88},  Lemma~7.43, p.~93, implies that
$$
M( f( \Gamma ) ) \geq \tau_n\left(  \frac{ |x_2-x_0| }{ |x_1-x_0| }      \right) 
=  \tau_n\left(  \frac{r_2}{r_1}      \right) 
$$
for a positive decreasing function $\tau_n(t)$ defined for $t > 0$,  such that $\lim_{t\to \infty} \tau_n(t)=0$ and for all $s>1$ we have (\cite{Vuo88}, p.~90)
\begin{equation} \label{ll6}
( \log ( \lambda_n^2 s ) )^{1-n} \leq \tau_n(s-1)/ \omega_{n-1} \leq ( \log s )^{1-n}
\end{equation}
where $\lambda_n\in [4,2e^{n-1})$ is a constant depending only on $n$, defined in \cite{Vuo88}, pp.~88--89.

Thus
\begin{eqnarray*}
&{}&
\frac{ K_I  \omega_{n-1} } { ( 2\beta_0 )^{n-1} }  
\geq \tau_n\left(  \frac{r_2}{r_1}      \right) 
\geq  \omega_{n-1} ( \log (  \lambda_n^2 (1 + r_2/r_1 ) ) )^{1-n}  
\end{eqnarray*}
so that
\begin{eqnarray*}
&{}&
\log (  \lambda_n^2 (1 + r_2/r_1 ) ) \geq K_I^{1/(1-n)} 2\beta_0 .
\end{eqnarray*}
Hence
$$
 \frac{r_2}{r_1}    \geq -1 +  \lambda_n^{-2} e^{ K_I^{1/(1-n)} 2\beta_0 } . 
$$
If $  \beta\geq C_2$ and $\beta_0> \beta/2$, 
then 
$$ 
e^{ K_I^{1/(1-n)} 2\beta_0 } > e^{ K_I^{1/(1-n)} \beta } \geq 
e^{ C_2 K_I^{1/(1-n)}   } = 80 \lambda_n^4 
$$ 
so that
$ \lambda_n^{-2} e^{ K_I^{1/(1-n)} 2\beta_0 } > 80 \lambda_n^2 > 2$ and hence
\begin{equation} \label{t1}
 \frac{r_2}{r_1}    \geq  -1 +  \lambda_n^{-2} e^{ K_I^{1/(1-n)} 2\beta_0 } > (1/2)  \lambda_n^{-2} e^{ K_I^{1/(1-n)} 2\beta_0 } \geq 40 \lambda_n^2 > 40  .  
\end{equation} 
Since, in particular, $r_2 > 2 r_1$, we find that the  spherical ring domain 
$A(x_0,r_1,r_2)$ is contained in the domain between $U_0$ and $U_1$, that is, in 
$f(A_0)$. 

We  apply the same argument with $A_0$ replaced by
$A_1= A(a,1, e^{\beta_0})$. Since the bounded component of
${\mathbb R}^n \setminus f(\overline{A_1})$ contains $U_1$, we may use the same point $x_0$ and we find that with analogous notations, denoting the resulting $r_1$ and $r_2$ by $r_1'$ and $r_2'$.
Now
$$
 \lambda_n^{-2} e^{ K_I^{1/(1-n)} \beta_0 } 
 \geq \lambda_n^{-2} e^{ C_2 K_I^{1/(1-n)}/2 } 
 = \lambda_n^{-2} \sqrt{  80  \lambda_n^{4}  } = \sqrt{  80 } > 8 .
$$ 
So we have
\begin{equation} \label{t2}
 \frac{r_2'}{r_1'}    \geq  -1 + \lambda_n^{-2} e^{ K_I^{1/(1-n)} \beta_0  } 
 \geq (1/2)  \lambda_n^{-2} e^{ K_I^{1/(1-n)} \beta_0  }  > 4 .
\end{equation}
Hence the spherical ring domain 
$A_1'=A(x_0,r_1',r_2')$ is contained in $f(A_1)$.

We then apply the same argument with $A_0$ replaced by
$A_2=A(a,e^{-\beta_0}, 1 )$. We may again use the same point $x_0$ and we denote  the resulting $r_1$ and $r_2$ by $r_1''$ and $r_2''$.  We have
\begin{equation} \label{t3}
 \frac{r_2''}{r_1''}    \geq  (1/2)  \lambda_n^{-2} e^{ K_I^{1/(1-n)} \beta_0  } > 4 .
\end{equation}
It follows that
the  spherical ring domain 
$A_2'=A(x_0,r_1'',r_2'')$ is contained in $f(A_2)$.  Write
$$
\beta_1 =  K_I^{1/(1-n)} \beta_0 -\log (2 \lambda_n^2)  . 
$$

The same argument can be applied with any choice of $x_0\in U_1$. The individual radii such as $r_j$ will depend on $x_0$ but the inequalities (\ref{t1})--(\ref{t3}) remain valid. 

Let $\rho_1<\rho_2\leq \rho_3<\rho_4$ be the numbers $r_1',r_1'', r_2',r_2''$ listed in increasing order.
The spherical ring domains $A_1'$ and $A_2'$ are disjoint, have the same center $x_0$, and lie in the topological ring domain $A_0'$. Hence the spherical ring domain $A_3'= A(x_0,\rho_1,\rho_4)$ lies in $A_0'$. The point $f(z)=f(0)=0$ lies between $A_1'$ and $A_2'$. Hence
$|x_0-f(z)|=|x_0|  \geq \rho_2$. If $z_1$ is a boundary point of $D'$ in $U_0$, then 
$$
|z_1|=|f(z)-z_1| \geq \rho_4-\rho_3\geq  (e^{\beta_1} - 1 )   \rho_3 . 
$$

Recall that $a'\in \partial D'$ and $|a'|=d(f(z),\partial D') = d(0,\partial D')=1$. Suppose that $a'\in U_0$. Taking $z_1=a'$ and using the numbers $\rho_j$ for $1\leq j\leq 4$ arising from that choice, we get 
$(e^{\beta_1} - 1 )   \rho_3 \leq 1$. But there are points of $\partial D'$ in $U_1$ and the distance from $0$ to each of them is
$\leq 2 \rho_3 \leq 2/(  e^{\beta_1} - 1       ) <1$, a contradiction. Hence we cannot have $a'\in U_0$, so $a'\in U_1$. After noting that, we may apply the above arguments taking $x_0$ to be $a'$. The first argument on $A_0$ then shows that
$D'$ contains $A(a',r_1,r_2)$, and (\ref{t1}) holds. Since $\beta_0$ can be arbitrarily close to $\beta$, we see that we can choose $r_1$ and $r_2$ so that
$$
\frac{r_2}{r_1}    \geq -1 +  \lambda_n^{-2} e^{ K_I^{1/(1-n)} 2\beta } . 
$$

Suppose that $a'',b\in \partial D'$ and $|a''|=1$. Since $|a''|=1$, we must have $a''\in U_1$.  If $b\in U_1$ then
$|a''-b| \leq 2 \rho_1<2\rho_3< 2/(  e^{\beta_1} - 1       )<1$ so that
$$
| \log |a''-b| | > \log  ( (e^{\beta_1} - 1 ) /2 )  .  
$$
Suppose that $b\in U_0$. Note that
$1=|a'-f(z)| \leq \rho_3$.  Then 
$$ 
|a''-b| \geq |b| - |a''| \geq  \rho_4-\rho_3 -1 \geq   (e^{\beta_1} - 1 ) \rho_3 -1
\geq   e^{\beta_1} - 2   
$$  
since $|b|=|b-f(z)| \geq  \rho_4-\rho_3$. Hence
$$
| \log |a''-b| | = \log |a''-b|  > \log   (e^{\beta_1} - 2 )   .  
$$
Now
$$
\frac{ e^{ K_I^{1/(1-n)} \beta } } {  4  \lambda_n^2 } 
\geq \frac{ e^{ K_I^{1/(1-n)} C_2 } } {  4  \lambda_n^2 } 
\geq \frac{  80 \lambda_n^4  } {  4  \lambda_n^2 } = 20  \lambda_n^2 > 2
$$
so that
$$
 \frac{    e^{ K_I^{1/(1-n)} \beta } /(2  \lambda_n^2 ) -  2  } {    2    }  = 
-1 +  \frac{    e^{ K_I^{1/(1-n)} \beta }    } {   4  \lambda_n^2    }
\geq    \frac{    e^{ K_I^{1/(1-n)} \beta }    } {   8 \lambda_n^2    }  .
$$
It follows that
$$
\beta' \geq \log  ( (e^{\beta_1} - 2 ) /2 )
= \log \frac{    e^{ K_I^{1/(1-n)} \beta } /(2  \lambda_n^2 ) -  2  } {    2    }  
\geq K_I^{1/(1-n)} \beta - \log (8 \lambda_n^2 ) .  
$$

Hence for $\beta> C_2$ we get
\begin{eqnarray*}
&{}&
\frac{ 1 + \beta } { 1 + \beta' } \leq
\frac{ 1 + \beta } { 1 + K_I^{1/(1-n)} \beta - \log (8  \lambda_n^2 )}
\\ &{}&
=  K_I^{1/(n-1)} +  K_I^{1/(n-1)} \frac{ 1 +  K_I^{1/(n-1)} (  \log (8  \lambda_n^2 )  - 1 )    } { \beta + K_I^{1/(n-1)} (  1  - \log (8  \lambda_n^2 )    ) }
\end{eqnarray*}
so that we obtain (\ref{b111}) with 
 \begin{equation} \label{bbb}
C_1 = 
 K_I^{1/(n-1)} +  K_I^{1/(n-1)} \frac{ 1 +  K_I^{1/(n-1)} (  \log (8  \lambda_n^2 )  - 1 )    } { C_2 + K_I^{1/(n-1)} (  1  - \log (8  \lambda_n^2 )    ) } .
\end{equation}
Choosing $C_2$ as in (\ref{c2}) we get
\begin{eqnarray*}
&{}&
C_1 = 
 K_I^{1/(n-1)} +  \frac{ 1 +  K_I^{1/(n-1)} (  \log (8  \lambda_n^2 )  - 1 )    } {  \log (80  \lambda_n^4 )  +  (  1  - \log (8  \lambda_n^2 )    ) } 
\\ &{}&
= K_I^{1/(n-1)} +  \frac{ 1 +  K_I^{1/(n-1)} (  \log (8  \lambda_n^2 )  - 1 )    } {   1  + \log 10 + 2 \log \lambda_n     } 
\end{eqnarray*}
To get an expression depending only on $K$ rather than $K_I$, so that the same expression is valid also for $f^{-1}$, we may replace $K_I$ by $K$ and set
\begin{equation*}
C_1  
= K^{1/(n-1)} +  \frac{ 1 +  K^{1/(n-1)} (  \log (8  \lambda_n^2 )  - 1 )    } {   1  + \log 10 + 2 \log \lambda_n     } .
\end{equation*}

This proves Theorem~\ref{th-qc1}. 

{\bf Remark.} We see from (\ref{bbb}) that by making a sufficiently large choice for $C_2$, depending only on $n$ and $K$, we can get $C_1$ (for the smaller range $\beta> C_2$) to be arbitrarily close to 
$K_I^{1/(n-1)}$. The trivial bound $C_1=C_2+1$ for $\beta \leq C_2$ is crude but we do not try to get a smaller bound here. On the basis of general principles, the proper order of magnitude for $C_1$ should be $K_I^{1/(n-1)}$.

\section{Proof of Theorem~\ref{th-qc2}.} \label{s6}

\subsection{A lemma.} The following auxiliary result shows that the $\lambda-$metrics do not vary too much in certain balls. 

\begin{lemma} \label{le8}
Let $D$ be a domain in ${\mathbb R}^n$ with at least two boundary points. Suppose that $x_0\in D$, that $0<t<1$, and that  $y\in B^n(x_0, t d(x_0,\partial D) )$ (so that $y\in D$). Then
\begin{equation} \label{ratio}
\frac{1}{C'} \leq   \frac{ \lambda'_D(x_0)    } {    \lambda'_D(y)    }    \leq C'  
\end{equation}
where
\begin{equation} \label{c'}
 C'  =   \frac{  1 + \log \frac{  1  } {  1-t     }    } {  1-t     }    .     
\end{equation}
\end{lemma}

Due to Lemma~\ref{lem2}, similar results are valid for $\lambda_D$ and $\lambda''_D$ instead of $\lambda'_D$.  

\subsection{\bf Proof of Lemma~\ref{le8}.} 
Suppose that $y\in B^n(x_0, t d(x_0,\partial D) )$.  If $x' \in \partial D$ and $|x'-x_0|= d(x_0,\partial D)$, then $d(y,\partial D) \leq |x'-y| \leq |x'-x_0|+|x_0-y| \leq (1+t) d(x_0,\partial D)$. Similarly, if $x'' \in \partial D$ and $|x''-y|= d(y,\partial D)$, then $d(x_0,\partial D) \leq |x''-x_0| \leq |x''-y|+|y-x_0| \leq d(y,\partial D) + t d(x_0,\partial D)$ so that $d(x_0,\partial D)\leq d(y,\partial D)/(1-t)$. 

We will consider the ratio  $ \lambda'_D(x_0)    /    \lambda'_D(y)     $.
By Lemma \ref{lem1} there are minimizing points $a,b\in \partial D$ for  
$ \lambda'_D(x_0) $, and there are minimizing points $a' ,b' \in \partial D$ for  
$ \lambda'_D(y) $. By (\ref{met2}) we have
\begin{equation} \label{min1}
 \frac{ 1}{\lambda'_D(x_0) } =  |x_0-a| \left(  1 + \left| \log \frac{  |a-b|    } {  |x_0-a|   }  \right|  \right)   ,
\end{equation} 
\begin{equation} \label{min2}
 \frac{ 1}{\lambda'_D(y) } \leq  |y -a| \left(  1 + \left| \log \frac{  |a-b|    } {  |y-a|   }  \right|  \right)   ,
\end{equation} 
\begin{equation} \label{min3}
 \frac{ 1}{\lambda'_D(y) } =  |y-a' | \left(  1 + \left| \log \frac{  |a' -b' |    } {  |y-a' |   }  \right|  \right)   ,
\end{equation} 
\begin{equation} \label{min4}
 \frac{ 1}{\lambda'_D(x_0) } \leq  |x_0-a' | \left(  1 + \left| \log \frac{  |a' -b' |    } {  |x_0-a' |   }  \right|  \right)   .
\end{equation} 

We naturally have
$$\min \{ |y-a|, |y-a'| \}  \geq d(y,\partial D)$$ and $$ \min \{|x_0-a|,  |x_0 -a' |\}  \geq d(x_0,\partial D) . $$ 

We have
$$
|y -a| \leq |y-x_0| + |x_0-a| \leq t d(x_0,\partial D) + |x_0-a| \leq (1+t) |x_0-a| 
$$
and
$$
|y -a| \geq |x_0-a| - |x_0-y| \geq |x_0-a| -  t d(x_0,\partial D) 
\geq (1-t)  |x_0-a| . 
$$
Hence
$$
(1-t)  \frac{  |a-b|    } {   |x_0-a|   } \leq  \frac{  |a-b|    } { (1+t) |x_0-a|   } \leq  \frac{  |a-b|    } {  |y-a|   } \leq  
 \frac{  |a-b|    } { (1-t)  |x_0-a|   } .
$$
It follows that
$$
      \left| \log \frac{  |a-b|    } {  |y-a|   }  \right|  \leq \left| \log \frac{  |a-b|    } {  |x_0-a|   }  \right| + \log \frac{1}{1-t} . 
$$
Thus
$$
 1 + \left| \log \frac{  |a-b|    } {  |y-a|   }  \right| \leq \left( 1 +  \log \frac{1}{1-t} \right) 
\left(  1 + \left| \log \frac{  |a-b|    } {  |x_0-a|   }  \right|  \right) . 
$$
Finally,
\begin{eqnarray} \label{min5}
&{}&
\frac{ 1}{\lambda'_D(y) } \leq (1+t) |x_0-a|
\left( 1 +  \log \frac{1}{1-t} \right) 
\left(  1 + \left| \log \frac{  |a-b|    } {  |x_0-a|   }  \right|  \right) \notag
\\ &{}& 
= (1+t) \left( 1 +  \log \frac{1}{1-t} \right) \frac{ 1}{\lambda'_D(x_0) } 
\notag
\\ &{}& 
\leq (1-t)^{-1}  \left( 1 +  \log \frac{1}{1-t} \right) \frac{ 1}{\lambda'_D(x_0) }  .
\end{eqnarray} 

An analogous argument shows first that
\begin{eqnarray} \label{min5a}
&{}&
|x_0 -a'| \leq |y-x_0| + |y-a' | \leq t d(x_0,\partial D) + |y-a' | \notag
\\ &{}& 
\leq 
t (1-t)^{-1} d(y,\partial D) + |y-a' | \notag
\\ &{}& 
\leq t (1-t)^{-1} |y-a' | + |y-a' |
= (1-t)^{-1} |y-a' | .
\end{eqnarray}
Next,
\begin{eqnarray} \label{min5b}
&{}&
|y-a' | \leq |y-x_0| + |x_0-a'| \leq t d(x_0,\partial D) + |x_0-a'|  \notag
\\ &{}&
\leq t |x_0-a'| + |x_0-a'| = (1+t) |x_0-a'| 
\end{eqnarray}
so that
$$
 |x_0-a'| \geq  (1+t)^{-1}  |y-a' | .
$$

Hence
$$
(1-t) \frac{  |a' -b' |    } {  | y -a'|   }  \leq 
\frac{  |a' -b' |    } {  |x_0 -a' |   }  \leq  
(1+t) \frac{  |a' -b' |    } {  |y-a' |   } 
 \leq  
(1-t)^{-1} \frac{  |a' -b' |    } {  |y-a' |   }  .
$$
It follows that
$$
\left| \log \frac{  |a' -b' |   } {  |x_0 -a'|   }  \right|  \leq \left| \log \frac{  |a' -b' |    } {  |y-a' |   }   \right| + \log \frac{1}{1-t} . 
$$
Thus
$$
 1 + \left| \log \frac{ |a' -b' |  } {  |x_0 -a'|   }  \right| \leq \left( 1 +  \log \frac{1}{1-t} \right) 
\left(  1 + \left| \log \frac{  |a' -b' |    } {  |y-a' |   }   \right|  \right) . 
$$
Finally,
\begin{eqnarray} \label{min6}
&{}&
\frac{ 1}{\lambda'_D(x_0) } \leq (1-t)^{-1}  | y -a' |
\left( 1 +  \log \frac{1}{1-t} \right) 
\left(  1 + \left| \log \frac{  |a' -b' |    } {  | y -a' |   }  \right|  \right) \notag
\\ &{}& 
= (1-t)^{-1}  \left( 1 +  \log \frac{1}{1-t} \right) \frac{ 1}{\lambda'_D(y) } .
\end{eqnarray} 

Combining these results we obtain (\ref{ratio}) with $C'$ given by (\ref{c'}). This proves Lemma~\ref{le8}.

\subsection{\bf Proof of Theorem~\ref{th-qc2}.}

Let the assumptions of  Theorem~\ref{th-qc2} be satisfied. 
We fix $x_0\in D$ and consider an arbitrary $y\in D$ to estimate $d'_D(f(x_0),f(y))$.
 Because of Lemma~\ref{lem2}, we may consider $d'_D$ instead of~$d_D$.
 
So let $x_0\in D$ be chosen. By Theorem~12.17, p.~159 in \cite{Vuo88} (taking $p=1$ there), there is $\theta'\in (0,1)$ depending on $n$ and $K$ only such that for all $y\in B^n(x_0, (1/2) d(x_0,\partial D) )$ we have 
$d(f(y), \partial D') \geq \theta' d(f(x_0),\partial D')$. (This can be applied to $f^{-1}$ also since if $f$ is $K-$quasi\-conformal, then so is $f^{-1}$.) 
Then, by Lemma~12.2 and Corollary 12.3 on page 154 in \cite{Vuo88}, there is $\theta\in (0,1/2)$ depending on $n$ and $K$ only such that for all $y\in B^n(x_0, \theta d(x_0,\partial D) )$ we have $|f(y)-f(x_0)|\leq (1/2) d(f(x_0),\partial D')$. 
Note that when verifying that the assumptions of \cite{Vuo88},  Lemma~12.2 are satisfied, we use the fact that in the notation of that lemma, we have 
$|f(y)-z| \geq d(f(y),\partial D')$. We will use this $\theta$ below.

In addition, we note that the choice of the number $1/2$ in the expression $(1/2) d(x_0,\partial D)$ in  \cite{Vuo88}, Lemma~12.2 and Corollary 12.3, is arbitrary and the method of proof immediately gives also the following result. 
Suppose that $t_0\in (0,1/2]$. Then there is $t_1\in (0,1/2]$ such that for all
$y$ with $|y-x_0|\leq t_1 d(x_0, \partial D)$ we have
$|f(y)-f(x_0)| \leq t_0 d(f(x_0), \partial D')$. Here $t_1$ depends only on $t_0$, $n$, and $K$. For small $t_0$, the rough order of magnitude of $t_1$ is $t_0^K$.  If $f$ is $K-$quasiconformal, then so is $f^{-1}$, so that this result can be applied to $f^{-1}$ to conclude that if $|f(y)-f(x_0)|\leq t_1 d(f(x_0), \partial D')$ then
$|y-x_0|\leq t_0 d(x_0, \partial D)$.

Choose $t_0=1/2$ and let $t_1$ correspond to this $t_0$. After that, use this $t_1$ in place of $t_0$ and find the corresponding $t_1$ and denote it by $t_2$. Write $\theta_1=\min\{t_2,\theta\} \in (0,1/2)$. Suppose first that $y\in \overline{B}^n(x_0, \theta_1 d(x_0,\partial D) ) $. Let $L$ be the line segment joining $f(x_0)$ to $f(y)$. By the above, $$f(y) \in B^n(f(x_0), t_1 d(f(x_0),\partial D') )$$ and hence $L \subset B^n(f(x_0), t_1 d(f(x_0),\partial D') ) $. For an arbitrary $z\in L$, write $z=f(w)$ ($w=f^{-1}(z)$ is unique since $f$ is a homeomorphism) and apply the above comments about $t_0$ and $t_1$ to $f^{-1}$
to conclude that $|w-x_0|\leq t_0 d(x_0, \partial D) = (1/2) d(x_0, \partial D) $ 
so that also $d(w, \partial D)  \leq (3/2) d(x_0, \partial D) $.
By Lemma~\ref{le8} we thus have
$$
1/C' \leq    \lambda'_D(w)/\lambda'_D(x_0) \leq C'
$$
where $C' = 2(1+\log 2)$ is as in (\ref{ratio}), taking $t=1/2$ in (\ref{ratio}). 

Using Theorem~\ref{th-qc1} (choosing an absolute constant $C_1$ that works for $\lambda'$)  we see that
$$
 \lambda'_{D'}(z) d(z,\partial D') \leq C_1  \lambda'_{D}(w) d(w,\partial D) 
$$
and so
\begin{eqnarray*}
&{}& d'(f(x_0),f(y)) \leq \int_{L} \lambda'_{D'}(z) |dz| 
= \int_{L} \frac{  \lambda'_{D'}(z) d(z,\partial D')  } {   d(z,\partial D')    } |dz| 
\\ &{}&
\leq C_1  \int_{L} \frac{  \lambda'_{D}(w) d(w,\partial D)  } {   d(z,\partial D')    } |dz| .
\end{eqnarray*}
We further have
$ \lambda'_{D}(w) d(w,\partial D) \leq (3/2) C' \lambda'_{D}(x_0) d(x_0,\partial D)$ for all these $w$, as well as
$d(z,\partial D') \geq (1/2)  d(f(x_0) ,\partial D') $ for all these $z$. Hence
\begin{eqnarray*}
&{}&
C_1  \int_{L} \frac{  \lambda'_{D}(w) d(w,\partial D)  } {   d(z,\partial D')    } |dz| 
\leq \frac{ 3 C_1 C' }{2} \frac{ 2 \lambda'_{D}(x_0) d(x_0,\partial D)  } {   d(f(x_0) ,\partial D')    } |f(y)-f(x_0)| . 
\end{eqnarray*}
By \cite{Vuo88}, (12.8), p.~156, we have
$$
\frac{ 2   |f(y)-f(x_0)|  } {   d(f(x_0) ,\partial D')    } 
\leq   2^{1-1/K} K  \left(  \frac{ |x_0-y|    } {  \theta_1   d(x_0,\partial D)   }         \right)^{\alpha}  
$$
where
$\alpha= K^{1/(1-n)}$.  
Thus
$$
d'(f(x_0),f(y)) \leq \frac{ 3 C_1 C' }{2}  \lambda'_{D}(x_0) d(x_0,\partial D)  2^{1-1/K} K  \left(  \frac{ |x_0-y|    } {  \theta_1   d(x_0,\partial D)   } \right)^{\alpha}  . 
$$

Pick $\varepsilon >0$. There is a rectifiable path $\gamma$ joining $x_0$ to $y$ in $D$ such that $\int_{\gamma} \lambda'_D(u) |du| < d'_D(x_0,y) + \varepsilon$. If all points of $\gamma$ lie in $\overline{B}^n (x_0, (1/2) d(x_0,\partial D) )$, set $x_1=y$. If not, let $x_1$ be the first point of $\gamma$ on $S^{n-1} (x_0, (1/2) d(x_0,\partial D) )$ encountered when tracing $\gamma$ from $x_0$ to $y$. Let $\gamma'$ be the part of $\gamma$ from $x_0$ to $x_1$. If $x_1=y$, we have
$\int_{\gamma'}  |du| \geq |x_0-y|$. Otherwise, we have
$\int_{\gamma'}  |du| \geq (1/2) d(x_0,\partial D) \geq |x_0-y|$. 
By Lemma~\ref{le8}, at all points $u$ of $\gamma'$ we have
$\lambda'_D(u) \geq \lambda'_D(x_0)/C'$ where $C' = 2(1+\log 2 )$. 
Hence
$$
d'_D(x_0,y) + \varepsilon > \int_{\gamma'} \lambda'_D(u) |du|
\geq (C')^{-1} \lambda'_D(x_0) |x_0-y|  .
$$ 
Since $\varepsilon >0$ is arbitrary, we get
\begin{equation} \label{ddd}
d'_D(x_0,y) \geq (C')^{-1} \lambda'_D(x_0) |x_0-y| .
\end{equation}  
This implies that
$$
\left(    \frac{  |x_0-y|  } { d(x_0,\partial D)  }  \right)^{\alpha} \leq 
\left(    \frac{  C' d'_D(x_0,y) } {  \lambda'_D(x_0) d(x_0,\partial D)  }  \right)^{\alpha}
$$
so that
$$
d'(f(x_0),f(y)) \leq \frac{ 3 \cdot 2^{-1/K} K C_1 C' }{  \theta_1^{\alpha} }   \lambda'_{D}(x_0) d(x_0,\partial D)   \left(    \frac{  C' d'_D(x_0,y) } {  \lambda'_D(x_0) d(x_0,\partial D)  }  \right)^{\alpha} . 
$$
The definition of $\lambda''_D(x_0)$ shows that
$ \lambda''_D(x_0) d(x_0,\partial D)  \leq 1$. With $C_0$ as in Lemma~\ref{lem2}, we thus have  $ \lambda'_D(x_0) d(x_0,\partial D)  \leq C_0$.   Hence
\begin{eqnarray*} 
&{}&
d'(f(x_0),f(y)) \leq \frac{ 3 \cdot 2^{-1/K}  K C_1 (C')^{1+\alpha} }{  \theta_1^{\alpha} }   ( \lambda'_{D}(x_0) d(x_0,\partial D) )^{1-\alpha}  d'_D(x_0,y)^{\alpha} 
\\ &{}&
\leq \frac{ 3 \cdot 2^{-1/K}  K C_1 (C' )^{1+\alpha} C_0^{1 - \alpha}  }{  \theta_1^{\alpha} }     d'_D(x_0,y)^{\alpha}  . 
\end{eqnarray*}
This gives (\ref{qc2-1}) for $d'_D$ with $C_2 =  3 \cdot 2^{-1/K}  K C_1 (C' )^{1+\alpha} C_0^{1 - \alpha}  \theta_1^{-\alpha} $.

Next, we follow the general ideas in the proof of Theorem~12.5 in \cite{Vuo88}, p.~155, with certain necessary changes since we are dealing with a different metric than the one  in \cite{Vuo88}.  

We have covered the case when $y\in B^n(x_0, \theta_1 d(x_0,\partial D) ) $. Suppose now that $y\notin B^n(x_0, \theta_1 d(x_0,\partial D) ) $. Choose a small $\varepsilon>0$.  There is a rectifiable closed Jordan arc $\gamma$ in $D$ joining $x_0$ to $y$ such that the length of $\gamma$ in the $\lambda'_D$-metric is less than $d'_D(x_0,y) + \varepsilon$. First find points $x_j$ for $1\leq j\leq p+1$ as follows. 
Suppose that $x_0, \dots , x_j$ have been chosen. If the subarc of $\gamma$ from $x_j$ to $y$ is contained in $\overline{B}^n(x_j, \theta_1 d(x_j,\partial D) ) $, set $p=j$, $x_{p+1}=y$ and stop. If not, let $x_{j+1}$ be the first  point  of $\gamma$
 on $S^{n-1}(x_j, \theta_1 d(x_j,\partial D) )   $ when we traverse $\gamma$ from $x_j$ to $y$. 
Since $\gamma$ is a compact subset of $D$, there is $A>0$ such that $d(x,\partial D)\geq A$ for all $x$ on $\gamma$. Hence we always go forward by at least a fixed Euclidean length of subarcs of $\gamma$, and so the process stops after finitely many steps.  It follows from 
the argument we used above to obtain (\ref{ddd})  that
$$
d'_D(x_j,x_{j+1} ) \geq (C')^{-1} \lambda'_D(x_j) \theta_1 d(x_j,\partial D)  .  
$$
Next, by Theorem~\ref{th-qc1}, we have
$$
\lambda'_{D'}(f(x_j) )  d(f(x_j),\partial D')  \leq 
C_1 \lambda'_D(x_j)  d(x_j,\partial D)  .  
$$
Hence
$$
d'_D(x_j,x_{j+1} ) \geq (C_1 C')^{-1} \lambda'_{D'}(f(x_j) ) \theta_1 d(f(x_j),\partial D')  . 
$$
For all $x\in \gamma$ between $x_j$ and $x_{j+1}$, we have
$$
|x-x_j| \leq \theta_1 d(x_j,\partial D)
$$
so that by the definition of $\theta_1$,
$$
|f(x)-f(x_j)| \leq t_1 d(f(x_j),\partial D')  \leq (1/2) d(f(x_j),\partial D') 
$$
and so
$$
d(f(x),\partial D') \geq (1/2) d(f(x_j),\partial D') .
$$
Now by Lemma~\ref{le8} we have
$$
 \lambda'_{D'}(f(x_j) ) \geq (1/C')  \lambda'_{D'}(f(x) ) . 
$$
Let $L$ be the line segment between $f(x_j)$ and $f(x_{j+1})$, hence contained in $\overline{  B^n}( f( x_j ) , (1/2) d(f(x_j),\partial D') )   $.   Then by Lemma~\ref{le8}
\begin{eqnarray*}
&{}&
d'_{D'}( f(x_j),f(x_{j+1}) ) \leq \int_{L} \lambda'_{D'}(z ) |dz| \leq 
  C'    \lambda'_{D'}(f(x_j) )  |   f(x_j) - f(x_{j+1})   | .
\end{eqnarray*}
Combining the above inequalities we obtain
$$
d'_{D'}( f(x_j),f(x_{j+1}) ) \leq
\frac{ (C')^2 C_1 } { 2 \theta_1 }    d'_D(x_j,x_{j+1} ) .
$$
Thus
\begin{eqnarray*}
&{}&
d'_{D'}( f(x_0),f(y) ) \leq \sum_{j=0}^p d'_{D'}( f(x_j),f(x_{j+1}) )
\\ &{}&
 \leq 
 \frac{ (C')^2 C_1 } { 2 \theta_1 }  \sum_{j=0}^p  d'_D(x_j,x_{j+1} )
 \leq  \frac{ (C')^2 C_1 } { 2 \theta_1 }   ( d'_D(x_0, y  ) + \varepsilon )  .  
\end{eqnarray*}
Since $ \varepsilon >0$ is arbitrary, we have
$$
d'_{D'}( f(x_0),f(y) ) \leq \frac{ (C')^2 C_1 } { 2 \theta_1 }    d'_D(x_0, y  ) .  
$$
This gives (\ref{qc2-1}) with $C_2= ((C')^2 C_1)/( 2 \theta_1 ) $ in the case when
$y\notin B^n(x_0, \theta_1 d(x_0,\partial D) ) $. The proof of  Theorem~\ref{th-qc2} is complete.

\end{document}